\documentclass[10pt]{article} % % 16/12/2017
\usepackage{amssymb}
\usepackage{xcolor}
\usepackage{amsmath,amsfonts,amsthm}

\usepackage{comment}
\usepackage{graphicx}
\graphicspath{ {./images/} }
\usepackage[normalem]{ulem}

\newtheorem{theorem}{Theorem}[section]
\newtheorem{lemma}[subsection]{Lemma}

\newtheorem{corollary}[theorem]{Corollary}
\newtheorem{remark}[theorem]{Remark}
\newtheorem{definition}[theorem]{Definition}
\newtheorem{example}[theorem]{Example}
\newtheorem{proposition}{Proposition}

\usepackage{fullpage}

\begin{document}
	
	\numberwithin{equation}{section}

	\title{Mass and topology of a static stellar model }
	
	\author{Maria Andrade$^1$, Benedito Leandro$^2$ and Thamara Policarpo$^3$}
	\footnotetext[1]{Universidade Federal de Sergipe, DMA,
		s/n, CEP 49100-000, S\~ao Crist\'ov\~ao, SE, Brazil. \textsf{maria@mat.ufs.br }}
	\footnotetext[23]{Universidade Federal de Goi\'as, IME,
		131, CEP 74001-970, Goi\^ania, GO, Brazil.}
\footnotetext[2]{Benedito Leandro was partially supported by CNPq/Brazil Grant 303157/2022-4. \textsf{bleandroneto@gmail.com}}
	\footnotetext[3]{Thamara Policarpo was partially supported by PROPG-CAPES [Finance Code 001]. \textsf{thamaraufg@discente.ufg.br}}
	\footnotetext[123]{The authors were partially supported by CNPq Grant 403349/2021-4.}
	
	\date{}

	\maketitle{}
	
	\begin{abstract}
		This study investigates the topological implications arising from stable (free boundary) minimal surfaces in a static perfect fluid space while ensuring that the fluid satisfies certain energy conditions. Based on the main findings, it has been established the topology of the level set $\{f=c\}$ (the boundary of a stellar model), where $c$ is a positive constant and $f$ is the static potential of a static perfect fluid space. We prove a non-existence result of stable free boundary minimal surfaces in a static perfect fluid space. An upper bound for the Hawking mass for the level set $\{f=c\}$ in a non-compact static perfect fluid space was derived, and the positivity of Hawking mass is provided in the compact case when the boundary $\{f=c\}$ is a topological sphere. We dedicate a section to revisit the Tolman-Oppenheimer-Volkoff solution, an important procedure for producing static stellar models. We will present a new static stellar model inspired by Witten's black hole (or Hamilton's cigar).
	\end{abstract}
	
	\vspace{0.2cm} \noindent \emph{2020 Mathematics Subject
		Classification} : 53C21, 53C23, 83C05.
	
	\vspace{0.4cm}\noindent \emph{Keywords}: static space, perfect fluid, vacuum, minimal surface, free-boundary.
	
	%%%%%%%%%%%%%%%%%%%%%%%%%%%%%%%%%%%%%%%%

	\section{Introduction}
	The Einstein equation with perfect fluid as a matter field of a static space-time metric $\hat{g}=-f^{2}dt^{2}+g$ is given by
	\begin{eqnarray*}\label{1}
		(Ric_{\hat{g}})_{ij}-\frac{R_{\hat{g}}}{2}\hat{g}_{ij}=T_{ij},
	\end{eqnarray*} 
{where $T_{ij}=(\mu+\rho)u_iu_j+\rho\hat{g}_{ij}$ is the energy-momentum stress tensor of a perfect fluid, $Ric_{\hat{g}}$ and $R_{\hat{g}}$ are the Ricci tensor and the scalar curvature for the metric $\hat{g},$ respectively.}
  Moreover, the smooth functions $\mu$ and $\rho$ are known as the {\it energy density} and {\it pressure} of the perfect fluid, respectively. Here, $u_i$ is {an} unit timeline vector field. Understanding a specific space-time metric is crucial in both physics and mathematics. Examples of static perfect fluid space-times are in \cite{barboza2018}.
	
Static space-times are significant as solutions to the Einstein equations in general relativity. Perhaps one of the most important solutions of this kind is the Schwarzschild space-time. This solution represents a static gravitational isolated system but is commonly recognized as a static model for a black hole. Mathematically, the Schwarzschild solution can be saw as a solution for 
	\begin{eqnarray*}
		Ric_{\hat{g}}=0. 
	\end{eqnarray*}

	Based on findings in astronomy, it seems that the universe can be described as a space-time that houses a perfect fluid. Galaxies can be seen as ``molecules'' in a smoothed and averaged form. The galactic fluid density is mainly contributed by the mass of the galaxies, with a small pressure from radiation. It is referred to as a ``perfect" fluid because it lacks heat conduction terms and ``stress" terms related to viscosity, as explained in \cite[p. 341]{oneill1983}. The study of static perfect fluid spaces holds great importance in physics.
	
	From the warped product structure of the static metric, we can define a static perfect fluid space as follows. 
	
	\begin{definition}\label{def1}
		A Riemannian manifold $(M^n,\,g,\,f,\,\rho)$ is said to be a static perfect fluid space if there exist smooth functions $f>0$ and $\rho,\,\mu$ on $M^n$ satisfying the perfect fluid equations:
		
		\begin{equation}\label{eqstfp}
			fRic=\nabla^2 f+\frac{(\mu-\rho)}{(n-1)}fg
		\end{equation}
		and
		\begin{equation}\label{eq2}
			\Delta f=\dfrac{[(n-2)\mu+n\rho]}{(n-1)}f,
		\end{equation}
		where ${Ric}$ and ${\nabla^2}$ stand for the Ricci and Hessian tensors,  respectively. Moreover, $\Delta$ and $R$ are the Laplacian and the scalar curvature for $g$. Moreover, the scalar curvature of $g$ and the density function satisfies the identity  $$\mu=\frac{R}{2}.$$ 
	\end{definition}
	Based on the above definition, we can infer that 
	\begin{eqnarray}\label{traceless}
		f\mathring{R}ic=\mathring{\nabla}^2f,
	\end{eqnarray}
	where {$\mathring{A}$ is the traceless part of a symmetric two-tensor $A$, i.e., $\mathring{A}=A-\frac{1}{n}\text{trace}(A)g.$} The above structure can be useful in analyzing static perfect fluid equations; see \cite{coutinho}. Different energy conditions suit various situations when studying static perfect fluid spaces. Some of the most relevant ones are the following (see in \cite[p. 175]{carrol} and \cite[p. 219]{wald}):
	\begin{itemize}
		\item The Weak Energy Condition or WEC states that $\mu \ge 0$ and $\mu+\rho \ge0$. 
		\item The Null Energy Condition or NEC states that $\mu+\rho \ge0$. 
		\item The Dominant Energy Condition or DEC  states that  $\mu \ge |\rho| $. 
	\end{itemize}

	{It is well known that t}here exists an unsolved conjecture called the {\it fluid ball conjecture} concerning static perfect fluid spaces. It proposes that non-rotating stellar models are spherically symmetric. A stellar model can be considered a solution to the equations for a static perfect fluid space (Definition \ref{def1}). However, proving the spherical symmetry of such spaces is a difficult task. It involves a range of conjectures that depend on various factors, such as the equation of state, the extent of the fluid region, the asymptotic assumptions, and the topology of the space-time (cf. \cite{coutinho2021,massod1} and the references therein). Usually, in the fluid region $\rho>0$, and in the exterior vacuum $\mu=\rho=0$. The boundary $\partial M=\{f=c\},$ where $0\leq c<1,$ of the interior fluid region and vacuum is a level surface of $f$. It is well-known that any conformally flat static perfect fluid space must be spherically symmetric. However, we encounter difficulties classifying solutions due to the traceless attribute of the static perfect fluid equation \eqref{traceless}. As seen in \cite{barboza2018}, an infinite number of solutions exist for the conformally flat static perfect fluid space. We will discuss the classical TOV solutions in Section \ref{tov}. Hence, gaining a better understanding of the geometry and topology of such a space satisfying Definition \ref{def1} is crucial.

	Mathematically speaking, the equations for a static perfect fluid offer great flexibility, thanks to the functions $\mu$ and $\rho$. For instance, an interesting choice is $\mu+\rho=0$ everywhere on $M.$ In this case, we will obtain $R=2\mu$ constant (cf. Lemma \ref{lmasood_O}),
	\begin{eqnarray}\label{staticnonnullcosm}
	fRic=\nabla^{2}f+\frac{Rf}{n-1}g\quad\mbox{and}\quad\Delta f+\frac{Rf}{n-1}=0.	\end{eqnarray}
	The above system represents the static vacuum equations with a non-null cosmological constant $\Lambda=\frac{R}{(n-1)}$ (cf. \cite{ambrozio2017}). Classical exact solutions for the above system are the de Sitter space, the Nariai space, and the Schwarzschild-de Sitter space. The topology of three-dimensional static spaces $(M^3,\,g)$ satisfying \eqref{staticnonnullcosm} can be
	studied using area-minimizing surfaces that can be produced in mean-convex manifolds by direct variational methods. In \cite{ambrozio2017}, the author proved that locally area-minimizing
	surfaces exist in $M\backslash\partial M$ only in exceptional cases. Also, in \cite{massod2}, the author used the theory of stable minimal surfaces to obtain topological information about static perfect fluid spaces. 
	
	It is known that if $\mu=\rho=0$, then we return to a static vacuum Einstein space with a null cosmological constant. That is, $(M^n,\,g,\,f)$ must satisfy
 \begin{eqnarray}\label{vac}
    fRic=\nabla^{2}f\quad\mbox{and}\quad\Delta f=0.
 \end{eqnarray}
For the static vacuum Einstein space-time, the static potential $f$ must be zero at the boundary $\partial M$ (cf. Theorem 1 in \cite{huang2018}). The topology and geometric properties of $\partial M$ were widely explored in the literature because it is closely related to the event horizon, i.e., the boundary of a static black hole \cite{galloway1993}. It is well-known that $\{f=0\}$ is a minimal hypersurface in the spatial factor of a static vacuum Einstein space-time (cf. \cite{coutinho}).

	An extension of the concept of an event horizon is the so-called photon surface, which can be understood as photons moving in spirals around a central black hole or naked singularity at a fixed distance. The photon surface is also significant for understanding trapped null geodesics and ensuring dynamical stability in the context of the Einstein equations. Furthermore, the presence of photon surfaces is connected to the occurrence of relativistic images in the gravitational lensing scenario. See \cite{cederbaum2021} and the references therein. We will formulate the following definition by building upon the research conducted by Cederbaum and Galloway; see \cite[Definition 2.5 and Proposition 2.8]{cederbaum2021}. We will see that the following definition is closely related to the boundary of a fluid region for a static stellar model, which is the level sets of lapse function $f$. We say that $\Sigma^{n-1}\subset M^n$ is a {photon surface} of a static perfect fluid space $(M^n,\,g,\,f,\,\rho)$ if the level set $\{f=c\}\subseteq\Sigma$, $c\in(0,\,1)$ is a regular value of $f$ and $\partial M$ is totally umbilical surface concerning the induced metric.

 %We will establish a connection between the concept of boundary for stellar models and the photon surfaces of black holes simultaneously.

 \iffalse
 \begin{definition}\label{photonsphere}
		We say that a boundary of a static perfect fluid space $\partial M$ is a {\bf photon surface} if $\{f=c\}\subseteq\partial M$, $c\in(0,\,+\infty)$ is a regular value of $f$ and $\partial M$ is totally umbilical hypersurface concerning the induced metric.
	\end{definition}
\fi

	A static perfect fluid space-time is often used as a static stellar model. A static stellar model is expected to be spherically symmetric (i.e., locally conformally flat); see \cite{kunzle,massod1}. Moreover, we can see from \cite[Proposition 1]{leandro2019} that for any locally conformally flat, static perfect fluid space, $\{f=c\}$ must be totally umbilical and have constant mean curvature. So, the concept of photon surface comes naturally from the structure of a perfect fluid space.
	
	%By a body $Q$ we shall mean a compact three dimensional manifold with $\partial Q$ (which need not be connected) all lying in some Riemannian manifold $M$. We say that $Q$ is handlebody if it is a diffeomorphic copy of a body in $\mathbb{R}^{3}$ bounded by a smooth surface. If $Q$ is not a handlebody, then there exists a compact stable minimal surface $\Sigma$ in the interior of $Q$ (cf. \cite{frankel}). 
	
	%It is well-known that if some space-time neighborhood $U$ of a body $Q\subset M$ with mean-convex boundary is static and $Ric(K,\,K)>0$ for all nonvanishing null vectors $K$ along $Q$, then $Q$ is a handlebody.

	%Astronomical evidence also indicates that the universe can be modeled (in smoothed, averaged form) as a space-time containing a perfect fluid whose “molecules” are the galaxies. At present, the dominant contribution to the density of the galactic fluid is the mass of the galaxies, with a much smaller pressure due mostly to radiation (see \cite[p. 341]{oneill1983}). Considering such background, in \cite{leandroernanipina} the authors proved that $\mu=\rho=0$ in the set $\{f=0\}.$ This set is close related with the event horizon, the boundary of a black hole. 

	Our focus lies in revisiting the consequences of the existence of stable minimal (and constant mean curvature) surfaces immersed in a static perfect fluid space. In this work, we will conduct this study with a different approach, not assuming any asymptotic condition which is usually considered. Moreover, inspired by \cite{ambrozio2015} and \cite{montezuma2021}, we will explore the consequences of free boundary minimal surfaces on such static spaces. {Here, in this work, we will suppose that the boundary of a manifold is connected.}

 As we have said before, the boundary of the interior fluid region and vacuum is a level surface of $f$ (see \cite[p. 57]{massod1}). In the following result, we prove that if $(M^n,\,g,\,f,\,\rho)$ is a static perfect fluid and we have a closed, connected, minimal, stable hypersurface in this space, then $f$ must be constant in this hypersurface provided that an equation of state holds.

	\begin{theorem}\label{lemmaf}
		Let $(M^{n},\,g,\,f,\,\rho)$ be an $n$-dimensional static perfect fluid space. Let $\Sigma$ be a  closed, connected, stable minimal hypersurface in $M$. Suppose that one of the following situations occurs:
		\begin{enumerate}
			\item $f$ is zero on $\Sigma$ and $\mu\geq0$ and $\rho\geq0$.
			\item $f>0$ ($f<0$) on $\Sigma$ and $\mu+\rho\geq0$.
		\end{enumerate}
		Then, $\Sigma$ is totally geodesic. Moreover, if the second statement holds and $n=3$, $f$ is constant at $\Sigma$ if, and only if, 
  \begin{eqnarray*}
  2\pi\mathfrak{X}(\Sigma) +\int_{\Sigma}\rho = 0.
		\end{eqnarray*}
	 Here, $\mathfrak{X}(\Sigma)$ stands for the Euler number of $\Sigma$.
	\end{theorem}

	In what follows, we will prove a result showing the rigidity of a static perfect fluid space metric using the theory of stable minimal hypersurfaces. The theorem below is a generalization of \cite[Proposition 5]{huang2018}; see also \cite[Proposition 14]{ambrozio2017}.
	
	\begin{theorem}\label{propositionsplitting}
		Let $(M^{n},\,g,\,f,\,\rho)$, $n\geq3$, be an $n$-dimensional static perfect fluid space in which $\mu\geq0$ and $\rho\geq0$.  Let $\Sigma^{n-1}$ be a locally area-minimizing,  closed, connected hypersurface in $M$. Suppose $f>0$ on $\Sigma$. Then there is a subset $U$ of $M$ where $g$ is Ricci-flat and $\mu=\rho=0$.
	\end{theorem}

 If we consider, in addition, that the functions are analytic on $M$ and $n=3$, we get a strong result regarding the non-existence of a stable minimal surface on a static perfect fluid space. 

 \begin{corollary}\label{coro1}
    Let $(M^{3},\,g,\,f,\,\rho)$ be an $3$-dimensional static perfect fluid space in which $\mu\geq0$ and $\rho\geq0$. Suppose that $f>0$ on $M$, and $g$ and $f$ are analytic in $M$. Therefore, there is no closed and stable minimal surface on $M.$
 \end{corollary}

	As we can see, the equation of state assumed in the above result is too strong. Let us see what we can get if we assume a weak equation of state like the NEC. Significant topological implications can be drawn regarding the minimal surface in the three-dimensional scenario not appealing to any analytic condition over the metric as we did in the previous result. We have extended \cite[Proposition 14]{ambrozio2017} to static perfect fluid spaces and provide a different demonstration for it. Notably, our demonstration does not necessitate integrations, allowing us to infer Proposition 24 and 25 in \cite{ambrozio2017}.

 \begin{theorem}\label{topospf}
    	Let $(M^{3},\,g,\,f,\,\rho)$ be a three-dimensional static perfect fluid space satisfying the NEC. Let $\Sigma$ be a locally area-minimizing, compact (possible with $\partial\Sigma\neq\emptyset$), connected surface in $M$. Suppose $f>0$ at $\Sigma$. Then, there is a subset $U$ of $M$ which admits a foliation by surfaces of constant Gaussian curvature.

Moreover,  
\begin{eqnarray*}
2\pi\mathfrak{X}(\Sigma)=\mu|\Sigma|+\int_{\partial\Sigma}K^{g}da, 
\end{eqnarray*}
where $|\Sigma|$, $\mathfrak{X}(\Sigma)$ and $K^g$ stand for the area of $\Sigma$, the Euler number of $\Sigma$, and the geodesic curvature of $\partial\Sigma$, respectively.
\end{theorem}

\begin{remark}
The above theorem agreed with Theorem C in \cite{ambrozio2017}. Considering $\partial\Sigma=\emptyset$ we get 
\begin{eqnarray*}
2\pi\mathfrak{X}(\Sigma)=\mu|\Sigma|. 
\end{eqnarray*}
Therefore, the topology of $\Sigma$ depends on the density function $\mu$ (or scalar curvature $R$).

    We point out the conclusion of Theorem \ref{topospf} also agreed with Theorem \ref{lemmaf}. In fact, in the above theorem we have $\mu=-\rho$ at $\Sigma$, and if $\partial \Sigma=\emptyset$ we get
    \begin{eqnarray*}
2\pi\mathfrak{X}(\Sigma)+\int_{\Sigma}\rho=0.
\end{eqnarray*}
\end{remark}

 When considering $\mu=-\rho$, we are in the vacuum ambient with a non-null cosmological constant, which satisfies \eqref{staticnonnullcosm}. In this scenario, the set $f^{-1}(0)$ has a physical meaning. It is closely related to the event horizon of a static black hole. In a more general background, we can consider the boundary to be a photon surface, which we already know is represented by the set $f^{-1}(c)$, where $c$ is a non-negative constant.

Here, $\mathfrak{M}_{Haw}(\Sigma)$ stands for the Hawking quasi-local mass of a surface $\Sigma$ in $(M^3,\,g)$, i.e.,
\begin{eqnarray}\label{haw}
    \mathfrak{M}_{Haw}(\Sigma)=\sqrt{\dfrac{|\Sigma|}{16\pi}}\left(1 - \dfrac{1}{16\pi}\int_{\Sigma}H^2\right).
\end{eqnarray}
For example, {if} $\Sigma$ is a sphere in $\mathbb{R}^3$. {Then, by Gauss-Bonnet Theorem, we obtain
$$\int_{\Sigma}H^2\geq8\pi\mathfrak{X}(\Sigma)=16\pi,$$
since $H^2\geq4K$, where $K$ is the Gauss curvature.} 

Thus, the Hawking mass of a sphere in $\mathbb{R}^3$ is always nonpositive. This observation concludes that the concept of mass tends to underestimate the mass in a region.

When the ambient manifold $M^3$ has nonnegative scalar curvature, a classic result of Christodoulou and Yau {says} that a CMC sphere on $M^3$ must have nonnegative Hawking mass. Recent results concerning 
$\mathfrak{M}_{Haw}=0$ were given
by Sun \cite{sun2018} and by Shi, Sun, Tian and Wei \cite{shi2019}. In \cite{miao2020}, the authors proved that given a constant mean curvature surface $\Sigma$ that bounds a compact manifold with nonnegative scalar curvature (and satisfies some intrinsic conditions), it must have positive Hawking mass. As an interesting application, they proved an identity involving the Hawking mass and the Brown-York mass $\mathfrak{M}_{BY}$:
\begin{eqnarray*}
   \mathfrak{M}_{BY}(\Sigma)=\mathfrak{M}_{Haw}(\Sigma) + \int_{\Sigma}H_0  + \sqrt{\dfrac{|\Sigma|}{16\pi}}\left(1-\sqrt{\dfrac{|\Sigma|}{16\pi}}H\right)^2 - \sqrt{\dfrac{|\Sigma|}{4\pi}},
\end{eqnarray*}
where $H_0$ is the mean curvature of the isometric embedding of $\Sigma$ in $\mathbb{R}^3.$ The mean curvature $H$ of $\Sigma$ is a positive constant \cite[Remark 1.6]{miao2020}. Remember, the Brown-York mass is defined by
\begin{eqnarray*}
    \mathfrak{M}_{BY}(\Sigma)=\dfrac{1}{8\pi}\int_{\Sigma}(H_0 - H).
\end{eqnarray*}
An upper bound of $\mathfrak{M}_{BY}$ for CMC surfaces was first derived by Lin and Sormani
\cite{lin2016} for an arbitrary metric $g$ on $\Sigma$. An upper bound of $\mathfrak{M}_{BY}$ was derived by Cabrera Pacheco, Cederbaum, McCormick, and Miao, assuming the Gauss curvature of $\Sigma$ was positive \cite{cabrerap2017}.

As we can see, the Wilmore functional is a fundamental piece on the definition of Hawking mas, i.e.,
$$Wilmore=\dfrac{1}{4}\int_{\partial M}H^2.$$
  The theorem below can be extracted from Theorem 22 in \cite{tiarlos} and Theorem B in \cite{fang2019} (where the authors used different approaches). In what follows, $H$, $|\partial M|$ and $\mathfrak{X}(\partial M)$ stand for mean curvature, the area, and the Euler number of $\partial M$, respectively.

   \begin{theorem}\label{cruztheorem1}
    	Let $(M^3,\,g,\,f,\,\rho)$ be a compact three-dimensional static perfect fluid space with  boundary $\partial M=f^{-1}(c)$, where $c\in[0,\,+\infty)$. Suppose that $\kappa=|\nabla f|$ is constant at $\partial M$. Then, 
  \begin{eqnarray*}
   \sqrt{\dfrac{16\pi}{|\partial M|}}\mathfrak{M}_{Haw}(\partial M)      \leq \int_M f\left[|\mathring{R}ic|^2 - \frac{1}{6}\mu(\mu+3\rho)\right] +2\pi\kappa( 2- \mathfrak{X}(\partial M)).
    \end{eqnarray*}
    Equality holds if and only if $\partial M$ is totally umbilical.
\end{theorem}

\begin{remark}
    The condition of $|\nabla f|$ in the above theorem follows immediately if $c=0.$
\end{remark}

\begin{corollary}
       	Let $(M^3,\,g,\,f,\,\rho)$ be a compact three-dimensional static perfect fluid space with totally umbilical boundary $\partial M=f^{-1}(c)$ homeomorphic to a sphere, where $c\in[0,\,+\infty)$. Suppose that $\kappa=|\nabla f|$ is constant at $\partial M$ and $|\mathring{R}ic|^2 \geq \dfrac{1}{6}\mu(\mu+3\rho)$. Then, $\mathfrak{M}_{Haw}(\partial M)\geq0$.
\end{corollary}

 \begin{corollary}\label{cruztheorem}
    	Let $(M^3,\,g,\,f,\,\rho)$ be a compact three-dimensional static perfect fluid space such that $\mu=-\rho$ with  boundary $\partial M=f^{-1}(c)$, where $c\in[0,\,+\infty)$. Suppose that $\kappa=|\nabla f|$ is constant at $\partial M$. Then, 
  \begin{eqnarray*}
        2\pi\kappa\mathfrak{X}(\partial M) \leq \int_M f|\mathring{R}ic|^2 +\frac{\kappa}{3}\mu|\partial M|+\dfrac{\kappa}{4}\int_{\partial M}H^2.
    \end{eqnarray*}
    Equality holds if and only if $\partial M$ is totally umbilical.

    \iffalse 
    In particular, 
         \begin{eqnarray*}
    \mathfrak{M}_{Haw}(\partial M)\geq\sqrt{\dfrac{|\partial M|}{16\pi}}\left(\mu\dfrac{|\partial M|}{8\pi}+2-\mathfrak{X}(\partial M)\right).
    \end{eqnarray*}
    Moreover, the Hawking mass is non-negative if $\partial M$ is diffeomorphic to a sphere and $\mu\geq0$. The equality holds if $(M^n,\,g)$ is Einstein.
    \fi
\end{corollary}

An interesting consequence of Corollary \ref{cruztheorem} is obtained when considering $\partial M$ as a static horizon boundary, i.e., $f^{-1}(0) = \partial M$ (cf. \cite[Theorem 2, Theorem C and Theorem B]{ambrozio2017}). The de Sitter space is an Einstein manifold, the Nariai space has a parallel Ricci tensor but is not Einstein, and the Schwarzschild-de Sitter spaces of positive mass are locally conformally ﬂat but {does} not have a parallel Ricci tensor. In the three-dimensional case, those are all known examples of
compact, simply connected static triples with positive scalar curvature (cf. \cite[Theorem 1]{ambrozio2017}).

    In Theorem \ref{cruztheorem}, we can assume $c=0.$ In this case, $H=0$ (see \cite{coutinho}). Therefore, we have 
    \begin{eqnarray*}
        2\pi\mathfrak{X}(\partial M) \geq\frac{\mu}{3}|\partial M|.
    \end{eqnarray*}
    The equality holds if $(M^3,\,g)$ is an Einstein manifold (de Sitter space).

For the non-compact case, we have the next inequality involving the Hawking mass. The following theorem was inspired by Theorem \ref{cruztheorem1}. In \cite{miao2005}, the author proved a lower bound for the Hawking mass of the boundary of an asymptotically flat static vacuum space, assuming some additional conditions in the geometry of the boundary.

\begin{theorem}\label{cruztheorem007-c1}
		Let $(M^3,\,g,\,f,\,\rho)$ be a three-dimensional (compact or non-compact)  static perfect fluid space such that $\Sigma=f^{-1}(c)$ is compact. Suppose that $\kappa=|\nabla f|$ is a constant at $\Sigma$. Then,
    \begin{eqnarray*}
			2\sqrt{\dfrac{16\pi}{|\Sigma|}}\mathfrak{M}_{Haw}(\Sigma) \leq 2-\mathfrak{X}(\Sigma) + \frac{\kappa}{2\pi c}\int_{\Sigma}H- \dfrac{\rho_{0}}{2\pi}|\Sigma|,
		\end{eqnarray*}
  where $\mathfrak{X}(\Sigma)$ is the Euler number of $\Sigma.$ The equality holds if and only if $\Sigma$ is totally umbilical. Here, $\rho_0=\rho\big|_{\Sigma}$ is a constant.
	\end{theorem}

 As an immediate consequence, we get the following corollary.

 \begin{corollary}
		Let $(M^3,\,g,\,f)$ be a three-dimensional static vacuum space satisfying \eqref{vac} such that $\Sigma=f^{-1}(c)$. Then,
 \begin{eqnarray}\label{BYHAW}
			\mathfrak{M}_{Haw}(\Sigma)\leq \dfrac{1}{2}\left[2-\mathfrak{X}(\Sigma) + \dfrac{\kappa}{8c\pi}\int_{\Sigma}H\right]\sqrt{\dfrac{|\Sigma|}{16\pi}}.
		\end{eqnarray}
The equality holds if and only if $\Sigma$ is totally umbilical. Any constant mean curvature surface in the Schwarzschild space satisfies the equality in \eqref{BYHAW}.
	\end{corollary}

\begin{remark}
First, we remember that in the vacuum, the condition over $\kappa$ is trivially satisfied (cf. \cite{cederbaum2021}). It is important to emphasize that the above theorem holds true if $M$ is  
  compact or non-compact and not necessarily $\Sigma$ is the boundary. Costa et al. \cite{costa2023} recently provided a positive mass theorem for the Brown-York mass, considering a compact static perfect fluid space in which $f^{-1}(0)$ is the boundary. 

 We prove that equality holds in Theorem \ref{cruztheorem007-c1} for a constant mean curvature surface in the Schwarzschild space.
\end{remark}

We can derive an interesting result for a locally conformally flat perfect fluid space.

\begin{theorem}
   Let $(M^3,\,g,\,f,\,\rho)$ be a conformally flat three-dimensional static perfect fluid space with $\Sigma=f^{-1}(c)$, where $c$ is a regular value of $f$. Then,
 \begin{eqnarray}\label{eqqqoi}
2\pi\mathfrak{X}(\Sigma) =\left[H\left(\frac{1}{4}H-\frac{\kappa}{c}\right)-\rho_0\right]|\Sigma|,
\end{eqnarray}
where $H$, $\kappa$ and $\rho_{0}$ are constants.
In particular, if $(M^3,\,g,\,f,\,\rho)$ is an Einstein static perfect fluid space, we have
    \begin{eqnarray*}\label{ambrozio2}
2\pi\mathfrak{X}(\Sigma) = \left(\dfrac{1}{4}H^2+\dfrac{\mu}{3}\right)|\Sigma|,
\end{eqnarray*}
where $H$ and $\mu$ are constants.
\end{theorem}

\begin{remark}
    The equation \eqref{eqqqoi} holds true if we consider a Bach-flat static perfect fluid space. This is a consequence of the results proved by \cite{leandro2021} combined with Proposition \ref{propbene}. Moreover, if we consider $\rho_0=0$ at $\Sigma$ in \eqref{eqqqoi}, which we saw that this is expected in a static stellar model, we get
    \begin{eqnarray*}
        2\pi\mathfrak{X}(\Sigma) =\frac{1}{4}H\left(H-\frac{4\kappa}{c}\right)|\Sigma|.
    \end{eqnarray*}
    We can see that $H$ determines the topology of $\Sigma.$
\end{remark}

	A study on free boundary minimal surfaces on the Schwarzschild solution was conducted in 2021 by Montezuma \cite{montezuma2021} and Barbosa-Espinar \cite{barbosa2021}. It is widely acknowledged that the Schwarzschild space adheres to conformal flatness. An important application of this property was observed in the computation of the Morse index of planes through the origin on the Schwarzschild space, as stated in Theorem 1.1 of \cite{montezuma2021}. The Schwarzschild solution's conformally flat structure was also critical in \cite{barbosa2021}. This highlights the significance of curvature conditions of the ambient space in deriving crucial insights about immersed free boundary minimal surface.
	Those results concerning the Schwarzschild solution and free boundary minimal surfaces combined with the ideas presented in this paper gave birth to the following theorem.

	\begin{theorem}\label{propositionsplitting1}
		Let $(M^{3},\,g,\,f,\,\rho)$ be a static perfect fluid space with constant scalar curvature, harmonic sectional curvature, and boundary $\partial M=f^{-1}(c)$, where $c$ is a regular value of $f$. There is no compact two-sided free boundary stable minimal surface $\Sigma^2\subset M^3$ with positive Gauss curvature such that 
  \begin{eqnarray*}
			\mu+\rho\leq \dfrac{2\kappa^2}{c^2},
\end{eqnarray*}
   where $\kappa=|\nabla f|\big|_{\partial\Sigma}$. 
	\end{theorem}

Consequently, we can infer the following result.
 \begin{corollary}
		Let $(M^{3},\,g,\,f,\,\rho)$ be a solution for a static vacuum space with a non-null cosmological constant, harmonic curvature, and boundary $\partial M=f^{-1}(c)$, where $c$ is a regular value of $f$. There is no compact two-sided free boundary stable minimal surface $\Sigma^2$ on $M^3$ with positive Gauss curvature.
	\end{corollary}

	 In \cite{ambrozio2015}, the author defined the following
	functional in the space of properly immersed surfaces
	\begin{eqnarray*}
		I(\Sigma)= \frac{1}{2}|\Sigma|\inf_{\Sigma}R+ |\partial\Sigma|\inf_{\partial\Sigma}H_{\partial M}.
	\end{eqnarray*}
	Ambrozio proved that 
	$$2\pi\mathfrak{X}(\Sigma)\geq I(\Sigma).$$

	The functional $I(\Sigma)$ played a crucial role in \cite{ambrozio2015}; see the references therein. In addition, the aforementioned outcome provides insight into the topology of an immersed stable minimal surface in a static perfect fluid space (see Schoen and Yau's Theorem 1 in \cite{ambrozio2015}).

 \begin{corollary}
		Let $(M^{3},\,g,\,f,\,\rho)$ be a solution for a static vacuum space with non-null cosmological constant, harmonic curvature, and boundary $\partial M=f^{-1}(c)$, where $c$ is a regular value of $f$. Let $\Sigma^2$ on $M^3$ be compact two-sided free boundary stable minimal surface. Then,
  \begin{eqnarray*}
	2\pi\mathfrak{X}(\Sigma)\geq \left(H_{\partial M}+\frac{\kappa}{c}\right)|\partial\Sigma| - \beta|\Sigma|,
\end{eqnarray*}
where $\beta=|\min_{\Sigma}K^{\Sigma}|$ and $K^\Sigma$ is the Gauss curvature of $\Sigma.$
	\end{corollary}

	%%%%%%%%%%%%%%%%%%%%%%%%%%%%%%%%%%%%%%%%%%%%%%%%%%%%%%%%%%%%%%%%%%%%%%%%%%%%%%%%%%%%%%%%%%%%%%%%%%%%%%%%%%%%%%%%%%%%%%%%%%%%%%%%%%%%%%%%%%%%%%%%%%%%%%%%%%%%%%%%%%%%%%%%%%%%%%%%%%%%%%%%%%%%%%%%%%%%%%%%%%%%%%%%%%%%%%%%%%%%%%%%%%%%%%%%%%%%%%%%%%%%%%%%%%%%%%%%%%%%%%%%%%%%%%%%%%%%%%%%%%%%%%%%%%%%%%%%%%%%%%

 \section{Tolman-Oppenheimer-Volkoff solution}\label{tov}

In this section, we will show the Tolman-Oppenheimer-Volkoff \cite{oppenheimer, tolman} procedure to obtain a stellar model. Then, we will provide some examples of exact solutions with physical interest.

Revisiting the TOV solution is crucial to understand the hypothesis assumed in our main results. Also, it is interesting to remember how to build examples of static perfect fluid spaces with a physical interpretation. Furthermore, we will see that there exist many possible static stellar models. The findings outlined in this paper are significant as they demonstrate the limitations that the boundary conditions of these solutions must adhere to, such as $\mu=\rho=0$ at the level sets of the lapse function (see also examples in \cite{barboza2018}).

  To avoid inconsistencies with the references \cite{gorini, oppenheimer, tolman, wyman} used in this section, we maintain the signature for the tensors used by them. We suppose that the universe is filled with a perfect fluid
with energy-momentum $T=8\pi[(\mu+\rho)u_iu_j-\rho\hat{g}_{ij}] $ and
consider a static spherically symmetric solution given by:
\begin{eqnarray*}
    \hat{g}= e^{v(r)}dt^2-e^{\gamma(r)}dr^2-r^2(d\theta^2+\sin^2(\theta)d\phi^2),
\end{eqnarray*}
where $r=x_1^2+x_{2}^2+x_3^2$. Here, $v(r)$ and $\gamma(r)$ are unknown functions of $r$. The spatial factor of such solutions is locally conformally flat; see \cite[Theorem 1]{vasquez}. For an isotropic sphere of fluid described by the above metric tensor the
pressure $\rho$ and density $\mu$ must satisfy the relations:
\begin{eqnarray}\label{tolman1}
    8\pi\mu(r)=\frac{1}{r^2}+e^{-\gamma(r)}\left(\frac{\gamma'(r)}{r}-\frac{1}{r^2}\right),
\end{eqnarray}

\begin{eqnarray}\label{tolman2}
    8\pi\rho(r)=-\frac{1}{r^2}+e^{-\gamma(r)}\left(\frac{v'(r)}{r}+\frac{1}{r^2}\right),
\end{eqnarray}
and the energy-momentum conservation equation
\begin{eqnarray}\label{tolman3}
    2\rho'(r)=-v'(r)(\rho+\mu).
\end{eqnarray}
 Since we have two unknown functions $\gamma$ and $v$,
the most satisfactory procedure would be to choose some physical boundary condition, say an equation of state involving $\rho$ and $\mu$, ensuring that the resulting solution would be of physical interest.

Consider the particular case $\gamma(r)=0$. {Then,} easily we can see that $\mu=0.$ Thus,
\begin{eqnarray*}
    \rho(r)=\dfrac{1}{2\pi r^2+c_1}.
\end{eqnarray*}
Therefore, 
\begin{eqnarray*}
    \hat{g}= e^{v(r)}dt^2-dr^2-r^2(d\theta^2+\sin^2(\theta)d\phi^2),
\end{eqnarray*}
where $e^{v(r)}=c_2\rho(r)^{-2}$ and $c_i\in\mathbb{R}$; $i=1,\,2.$

Considering $\mu(r)=c$ (Volkoff's massive sphere), where $c\neq0$ is a constant, from \eqref{tolman1} we get
\begin{eqnarray*}
    8c\pi r^2=1-rx'(r)-x(r)=1-[rx(r)]',
\end{eqnarray*}
where $x(r)=e^{-\gamma(r)}$. Hence,
\begin{eqnarray*}
   e^{-\gamma(r)} = 1-\dfrac{8c\pi}{3} r^2 +\frac{k}{r}.
\end{eqnarray*}
Here, $k$ is a constant and measure of the discrepancy between
the Newtonian and relativistic values for the mass of the sphere. Since $\mu$ is constant, from \eqref{tolman3} we can infer that
\begin{eqnarray*}
    8\pi(\rho(r)+c)=k_1e^{-v(r)/2};\quad k_1\in\mathbb{R}.
\end{eqnarray*}
Adding \eqref{tolman1} and \eqref{tolman2}, we have
\begin{eqnarray*}
    8\pi(\rho(r)+c)=e^{-\gamma}\left(\dfrac{\gamma'(r)}{r}+\dfrac{v'(r)}{r}\right),
\end{eqnarray*}
combining with the previous identity leads us to 
\begin{eqnarray*}
    k_1re^{-v(r)/2}=e^{-\gamma(r)}\left(\gamma'(r)+v'(r)\right).
\end{eqnarray*}
At this point, Equation 3.8 in \cite{wyman} missed the $r$ at the left-hand side of the above identity. Fortunately, this was only a typo. Then, 
\begin{eqnarray*}
    \frac{k_1}{2}re^{\gamma(r)}=[e^{v(r)/2}]' + \frac{\gamma'(r)}{2}e^{v(r)/2}.
\end{eqnarray*}
The solution for the above ODE is
\begin{eqnarray*}
    e^{v(r)}=e^{-\gamma(r)}\left(\frac{k_1}{2}\int re^{3\gamma(r)/2}dr+k_2\right)^2,
\end{eqnarray*}
where $k_2\in\mathbb{R}.$

The case $k=0$ is known as the Schwarzschild interior solution, which we will obtain explicitly below. The Schwarzschild exterior solution is obtained considering 
\begin{eqnarray*}
    e^{v(r)}=e^{-\gamma(r)}=1-\dfrac{2m}{r},
\end{eqnarray*}
where $m$ is the mass of the sphere. Therefore, any valid solution in the interior of the sphere must satisfy the following:
\begin{eqnarray*}
    e^{v(r_b)}=e^{-\gamma(r_b)}=1-\dfrac{2m}{r_b},
\end{eqnarray*}
where $r_b$ is the value of $r$ at the boundary of the sphere. Moreover, it is required that $\rho(r_b)=0$. 

Now, consider the case $\mu=\mu(\rho)$; we can proceed like in \cite{oppenheimer,tolman}. In this case, from \eqref{tolman3} we have
\begin{eqnarray*} v(r)=v(r_b)-2\int_{0}^{\rho}\dfrac{d\rho}{\mu(\rho)+\rho}.
\end{eqnarray*}
Considering the boundary conditions discussed above, we get
\begin{eqnarray*}
    e^{v(r)}=\left(1-\dfrac{2m}{r_b}\right)\exp{\left(-2\int_{0}^{r_b}\dfrac{d\rho}{\mu(\rho)+\rho}\right)}.
\end{eqnarray*}

By a change of variable $e^{-\gamma(r)}=1-\dfrac{2u(r)}{r}$, from \eqref{tolman1} we get $u'(r)=4\pi r^2\mu(r).$ Hence, the boundary condition yields to $$m(r)=4\pi \int_0^{r} s^2\mu(s)ds,$$
where $m(0)=0$. Now, using $e^{-\gamma(r)}=1-\dfrac{2m(r)}{r}$ and \eqref{tolman3} into \eqref{tolman2} we get the TOV equation (see \cite{oppenheimer}), i.e.,
\begin{eqnarray*}
      8\pi\rho(r)&=&-\frac{1}{r^2}+e^{-\gamma(r)}\left(\frac{v'(r)}{r}+\frac{1}{r^2}\right)\nonumber\\
      &=& -\frac{1}{r^2}+\left(1-\dfrac{2m(r)}{r}\right)\left(-\frac{2\rho'(r)}{r(\mu+\rho)}+\frac{1}{r^2}\right)\nonumber\\
      &=& -\frac{2\rho'(r)}{r^2(\mu+\rho)}\left(r-2m(r)\right) - \frac{2m(r)}{r^3}.
\end{eqnarray*}
Therefore,
\begin{eqnarray*}
    \rho'(r)=-\dfrac{(m+4\pi r^3\rho)(\mu+\rho)}{r(r-2m)}.
\end{eqnarray*}
Complementing the last equations with an equation of state
relating $\rho$ and $\mu$, one has a closed system of three equations
for the three variables $\rho$, $\mu$, and $m$.

For instance, consider the case in which the fluid is the Chaplygin
gas \cite{gorini} whose equation of state is
\begin{eqnarray*}
    \mu=\dfrac{-c^2}{\rho},
\end{eqnarray*}
where $c\in\mathbb{R}.$ Hence,
\begin{eqnarray*}
    \rho'(r)=\dfrac{(m+4\pi r^3\rho)(c^2-\rho)}{\rho r(r-2m)}\quad\mbox{and}\quad m'(r)=-4k_3^2\pi\frac{r^2}{\rho(r)}.
\end{eqnarray*}
The system admits two exact solutions with
constant pressure. One is $\mu=-\rho=-c$ and $m=\dfrac{4c\pi}{3}r^3$ and
\begin{eqnarray*}
    e^v(r)=e^{-\gamma(r)}=1-\frac{8\pi c}{3}r^2,
\end{eqnarray*}
which is the Schwarzschild interior solution. We can see that boundary conditions for the stellar model to be complete, i.e.,
 $r_b=\frac{3mc}{4\pi}$, is determinated by the value of the mass $m$ and the energy-density $c$. The second solution is the Einstein static universe, i.e., $\rho=-\frac{c\sqrt{3}}{3}$ and $\mu=c\sqrt{3}.$

Wyman presented a non-trivial example of physical interest \cite{wyman}. The author considered the density $$8\pi\mu=\frac{5}{R}r^2,$$ where $R$ is a constant. The idea was generalized in the Volkoff example, where we considered the density $\mu$ constant. The expression for the pressure is complicated and given by
$$8\pi\rho = -\dfrac{1}{R}r^2 + \dfrac{2}{R^2}\left(1-\frac{r^4}{R^4}\right)^{1/2}\coth\left[\frac{1}{2}\sin^{-1}\left(1-\frac{r^4}{R^4}\right)^{1/2}+B\right],$$
where $B$ is constant; see details at \cite{massod1987}.

To explicitly give the expression for the lapse function and the metric, we need to establish that $r_b^5=2mR^4,$ where $m$ stands for the mass. Therefore, 
\begin{equation*}
f(r) = \left\{
\begin{matrix}
A\sinh\left[2\sin^{-1}\left(1-\frac{r^4}{R^4}\right)^{1/2}+B\right] & ; & r < r_b,\\
\left(1-\dfrac{2m}{r}\right)^{1/2} & ; & r\geq r_b.
\end{matrix}
\right.
\end{equation*}
Here, $A$ is constant. The spatial factor of the metric is also divided into two parts, representing the interior and exterior of the perfect fluid stellar model, i.e.,

\begin{equation*}
g = \left\{
\begin{matrix}
\left(1-\dfrac{r^4}{R^4}\right)^{-1}dr^2+r^2(d\theta^2+\sin^2(\theta)d\phi^2) & ; & r< r_b,\\\\
\left(1-\dfrac{2m}{r}\right)^{-1}dr^2+r^2(d\theta^2+\sin^2(\theta)d\phi^2) & ; & r\geq r_b.
\end{matrix}
\right.
\end{equation*}
The constants $A,\,B$ and $R$ are related to the boundary conditions at $r=r_b$. 
\iffalse
Taking the particular case of $k=0$, from \eqref{tolman2} and \eqref{tolman3} we have
\begin{eqnarray*}
    \dfrac{-d\rho}{(\rho+c)(\rho+\frac{c}{3})}=\dfrac{4\pi dr}{1-\frac{8c\pi}{3}r^2}.
\end{eqnarray*}
Solve the above equation to obtain
\begin{eqnarray*}
   \frac{3}{2c}\log\left(\dfrac{\rho(r)+c}{\rho(r)+\frac{c}{3}}\right)=4\pi\dfrac{1}{2\sqrt{\frac{8c\pi}{3}}}\log\left(\dfrac{1+\sqrt{\frac{8c\pi}{3}}r}{1-\sqrt{\frac{8c\pi}{3}}r}\right)+\hat{k}.
\end{eqnarray*}
Hence,
\begin{eqnarray*}
    \rho(r)=-c\left[\dfrac{1-\dfrac{c_3}{3}\left(\dfrac{1+\sqrt{\frac{8c\pi}{3}}r}{1-\sqrt{\frac{8c\pi}{3}}r}\right)^{\frac{1}{2}\sqrt{\frac{8c\pi}{3}}}}{1-c_3\left(\dfrac{1+\sqrt{\frac{8c\pi}{3}}r}{1-\sqrt{\frac{8c\pi}{3}}r}\right)^{\frac{1}{2}\sqrt{\frac{8c\pi}{3}}}}\right].
\end{eqnarray*}
Again, from \eqref{tolman3} we have
\begin{eqnarray*}
    e^{v(r)}=\tilde{k}(\rho(r)+c)^{-2}
\end{eqnarray*}
which gives us
\begin{eqnarray*}
    \hat{g}= \dfrac{\tilde{k}dt^2}{(\rho(r)+c)^{2}}-\dfrac{dr^2}{1-\dfrac{8c\pi}{3} r^2 }-r^2(d\theta^2+\sin^2(\theta)d\phi^2).
\end{eqnarray*}
\fi

	%%%%%%%%%%%%%%%%%%%%%%%%%%%%%%%%%%%%%%%%%%%%%%%%%%%%%%%%%%%%%%%%%%%%%%%%%%%%%%%%%%%%%%%%%%%%%%%%%%%%%%%%%%%%%%%%%%%%%%%%%%%%%%%%%%%%%%%%%%%%%%%%%%%%%%%%%%%%%%%%%%%%%%%%%%%%%%%%%%%%%%%%%%%%%%%%%%%%%%%%%%%%%%%%%%%%%%%%%%%%%%%%%%%%%%%%%%%%%%%%%%%%%%%%%%%%%%%%%%%%%%%%%%%%%%%%

 \section{Background}
	
	Let us prove a well-known identity useful for studying static perfect fluid spaces; see Proposition 2 in \cite{coutinho}. 
	\begin{lemma}\label{lmasood_O}
		Let $\big(M^n,\,g)$ be the spatial factor of a static perfect fluid space-time. Then,
		\begin{eqnarray*}\label{modaso_O}
			-f\nabla\rho=(\mu+\rho)\nabla f.
		\end{eqnarray*}
	\end{lemma}
	
	\begin{proof}
		In a Riemannian manifold $M$ it is possible to relate the curvature with a smooth function using the Ricci identity:
		\begin{eqnarray}\label{ricciid}
			\nabla_{i}\nabla_{j}\nabla_{k}f-\nabla_{j}\nabla_{i}\nabla_{k}f=R_{ijkl}\nabla^{l}f.\nonumber
		\end{eqnarray}
		
		From \eqref{eqstfp}, it is easy to see that
		\begin{eqnarray*}
			\nabla_{i}fR_{jk}+f\nabla_{i}R_{jk}=\nabla_{i}\nabla_{j}\nabla_{k}f+\nabla_{i}hg_{jk},
		\end{eqnarray*}
		where $h=\frac{\mu-\rho}{n-1}f$.
		Thus, from \eqref{ricciid}, we get
		\begin{eqnarray*}
		\nabla_{i}fR_{jk}+f\nabla_{i}R_{jk}=\nabla_{j}\nabla_{i}\nabla_{k}f+R_{ijkl}\nabla^{l}f+\nabla_{i}hg_{jk}.
		\end{eqnarray*}
		Contracting the above identity under $i$ and $k$, and using $\mathrm{div}Ric=\frac{1}{2}dR$, we obtain
		\begin{eqnarray*}
			\frac{1}{2}f\nabla_{j}R=\nabla_{j}\Delta f+\nabla_{j}h.
		\end{eqnarray*}
		Now{,} from \eqref{eq2} we gather that
		\begin{eqnarray}\label{eq2-PJM}
			\nabla_{i}h=\frac{1}{n-1}\left(R\nabla_{i}f+\frac{1}{2}f\nabla_{i}R\right).
		\end{eqnarray}
				By the definition of $h$ the result follows.
	\end{proof}

In this work, we will use a couple of well-known results concerning the static perfect fluid structure of the metric. We can consult \cite{leandro2019} for more details.

\begin{lemma}\label{lemmaT}\cite[Lemma 1]{leandro2019}
    Let $(M^n,\,g,\,f)$ be a static perfect fluid space. Then,
    \begin{eqnarray*}
        fC_{ijk}=W_{ijkl}\nabla^lf + T_{ijk}
    \end{eqnarray*}
\end{lemma}

Here, $C$ and $W$ are the Cotton and Weyl tensors, respectively. Moreover, 
\begin{eqnarray*}
    T_{ijk}&=& \dfrac{(n-2)}{(n-1)}(R_{ik}\nabla_jf-R_{jk}\nabla_if) + \dfrac{R}{(n-2)}(\nabla_ifg_{jk}-\nabla_jfg_{ik})\nonumber\\
    &&-\dfrac{1}{(n-2)}(R_{il}\nabla^lfg_{jk}-R_{jl}\nabla^lfg_{ik}).
\end{eqnarray*}

The analysis of tensor $T$ is important for understanding the structure of the static perfect fluid space. The next result was inspired by Proposition 1 in \cite{leandro2019}. We highlight that if the scalar curvature is constant and the Weyl curvature is harmonic, then the sectional curvature is harmonic.

\begin{proposition}\label{propbene}
    Let $(M^3,\,g,\,f,\,\rho)$ be a locally conformally flat static perfect fluid space. Suppose that $\Sigma=f^{-1}(c)$, where $c$ is a regular value of $f$. Then, for any
local orthonormal frame $\{e_1,\,e_2,\,e_3\}$ with $e_1=-\dfrac{\nabla f}{|\nabla f|}$ and $\{e_2,\,e_3,\, e_4\}$ tangent to $\Sigma$, we have
\begin{enumerate}
    \item $\nabla f$ is an eigenvector of the Ricci operator. Moreover, either $\lambda=\lambda_1=\lambda_2=\lambda_3$ or $\lambda=\lambda_2=\lambda_3$ and $R_{11}=\lambda_1$. Here, $\lambda_i$ is the eigenvalue of $Ric$ associated to $e_i$.
    \item $|\nabla f|$ is constant at $\Sigma$.
    \item $\Sigma$ is totally umbilical with constant mean curvature.
\end{enumerate}
\end{proposition}
\begin{proof}
    Consider an orthonormal frame $\{e_1,\,e_2,\,e_3\}$ diagonalizing $Ric$ at $q\in\mathbb{M}$ with associated eigenvalues $\lambda_k(q)$, $k=1,\,2,\,3$, respectively. That is, $$R_{ij}(q) = \lambda_{j}\delta_{ij}(q).$$

    From Lemma \ref{lemmaT}, we have $T=0$. In particular, $T_{ijj} = 0$, i.e.,
    \begin{eqnarray*}
        (\lambda_j + 2\lambda_i - R)\nabla_jf=0.
    \end{eqnarray*}
    Assume that, for a fixed $j,$ $\nabla_jf\neq0$ and $\nabla_i f (p) = 0$ for all $i\neq j$; $i,\,j\in \{1,\, 2,\, 3\}$. Then, we have $Ric(\nabla f) = \lambda_j\nabla f$ , i.e., $\nabla f$ is a eigenvector for $Ric.$ Moreover, we know that $\lambda_j$ has multiplicity $1$ and $\lambda_i$ has multiplicity $3$. In the other case, if $\nabla_i f \neq 0$ for at least two distinct
directions, we concluded that $\lambda=\lambda_1=\lambda_2=\lambda_3$, and we also have $\nabla f$ an eigenvector for $Ric$.

From the above discussion we can take $\{e_1,\,e_2,\,e_3\}$ with $e_1=-\dfrac{\nabla f}{|\nabla f|}$ and $\{e_2,\,e_3,\, e_4\}$ tangent to $\Sigma$, diagonalizing $Ric$ at $q$. Hence, from the static equations, we get
\begin{eqnarray*}
    0=f\mathring{R}ic(\nabla f,\,e_a) = \dfrac{1}{2}\nabla_a|\nabla f|^2-\dfrac{\Delta f}{3}g(\nabla f,\,e_a).
\end{eqnarray*}
Thus, $\kappa=|\nabla f|$ constant at $\partial M.$

On the one hand, the second fundamental form of $\Sigma$ is 
\begin{eqnarray*}
    A_{ab}=-\dfrac{\nabla_a\nabla_b f}{|\nabla f|}.
\end{eqnarray*}
{On the other hand},
\begin{eqnarray*}
    \nabla_a\nabla_b f= f\mathring{R}_{ab}+\dfrac{\Delta f}{3}g_{ab}=\left(f\lambda-f\frac{R}{3}+\dfrac{\Delta f}{3}\right)g_{ab}.
\end{eqnarray*}

Therefore, 
\begin{eqnarray*}
    A_{ab}=\dfrac{H}{2}g_{ab},
\end{eqnarray*}
where 
\begin{eqnarray}\label{meanlcf}
    H=-\dfrac{2}{\kappa}\left(f\lambda-f\frac{R}{3}+\dfrac{\Delta f}{3}\right).
\end{eqnarray}

Contracting the Codazzi equation
\begin{eqnarray*}
    R_{1cab}=\nabla_{a}A_{bc}-\nabla_{b}A_{ac}
\end{eqnarray*}
gives us 
\begin{eqnarray*}
    0=R_{1a} = \dfrac{1}{2}\nabla_aH.
\end{eqnarray*}
We obtained $H$ constant at $\Sigma$.
\end{proof}

The next theorem is important to prove Theorem \ref{cruztheorem}.

\begin{theorem}[\cite{tiarlos}]
    Let $(M^n,\, g)$ be a compact Riemannian manifold with boundary $\partial M$ (possible empty). If $T$ is a symmetric $2$-tensor field such that $\textnormal{div}T=0,$ and $X$ is a vector field on $M,$ then
    \begin{eqnarray*}
        \int_{M}X(\tau)=\frac{n}{2}\int_M\langle \mathring{T},\,\mathfrak{L}_{X}g\rangle - n \int_{\partial M} \mathring{T}(X,\,N),
    \end{eqnarray*}
    where $\tau=\textnormal{trace}(T),$ $\mathring{T}$ and $N$ stand for the trace free part of $T$ and the outer unit normal along $\partial M.$ 
\end{theorem}

For an
$n$-dimensional Riemannian manifold $(M, g)$, the second Bianchi identity yields the
well-known divergence formulae 
\begin{eqnarray*}
\mathrm{div}Ric=\frac{1}{2}dR\quad\mbox{and}\quad \mathrm{div}Rm= - DRic,
\end{eqnarray*}
where $D$ is the first-order differential operator defined by 
\begin{eqnarray*}
Dw(X,\,Y,\,Z)=\nabla_{X}w(Y,\,Z)-\nabla_{Y}w(X,\,Z)
\end{eqnarray*}
for a two-form $w.$ Consequently, 
\begin{eqnarray*}
\mathrm{div}W= -(n-3)DSch,
\end{eqnarray*}
where $W$ is the Weyl tensor and $Sch$ stands for the Schouten tensor:
\begin{eqnarray*}
Sch(X,\,Y) = \frac{1}{n-2}\left(Ric(X,\,Y) - \frac{R}{2(n-1)}g(X,\,Y)\right).
\end{eqnarray*}

Using the ideas presented by \cite{hwang2014}, let us prove that if the curvature tensor is harmonic, then $Ric(\nabla f,\,\nabla f)$ is constant at $\partial M=f^{-1}(c)$, where $c$ is a regular value of $f$. This property will be important in the proof of Theorem \ref{propositionsplitting1}.

\begin{lemma}\label{interesting}
Consider a non-trivial static perfect fluid space $(M^n,\,g,\,f,\,\rho)$ such that the curvature tensor is harmonic and the scalar curvature is constant. For each regular value $c$ of $f$ and a tangent vector $U$ to $f^{-1}(c)$, we have
\begin{eqnarray*}
	Ric(\nabla f,\,U)=0
\end{eqnarray*}
on $f^{-1}(c)$. Consequently, $|\nabla f|$ and $\rho$ are constants at $f^{-1}(c)$.
\end{lemma}
\begin{proof}
We start by taking the derivative of \eqref{eqstfp} and using the Ricci identity:
\begin{eqnarray*}
	\nabla_{i}fR_{jk}-\nabla_{j}fR_{ik}+f(\nabla_{i}R_{jk}-\nabla_{j}R_{ik})=R_{ijkl}\nabla^{l}f+(\nabla_{i}hg_{jk}-\nabla_{j}hg_{ik}),
\end{eqnarray*}
where $h=\frac{\mu-\rho}{n-1}f.$
From the above equation and Lemma \ref{lmasood_O}{,} we get
\begin{eqnarray*}
	\nabla_{i}fR_{jk}-\nabla_{j}fR_{ik}+f(\nabla_{i}R_{jk}-\nabla_{j}R_{ik})&=&R_{ijkl}\nabla^{l}f+\frac{R}{n-1}(\nabla_{i}fg_{jk}-\nabla_{j}fg_{ik})\nonumber\\
	&+&\frac{f}{2(n-1)}(\nabla_{i}Rg_{jk}-\nabla_{j}Rg_{ik}).
\end{eqnarray*}
Now{,} the Cotton tensor 
\begin{eqnarray*}\label{cotton}
	C_{ijk}&=&\nabla_{i}R_{jk}-\nabla_{j}R_{ik}-\frac{1}{2(n-1)}\left(\nabla_{i}Rg_{jk}-\nabla_{j}Rg_{ik}\right) 
\end{eqnarray*}
gives us
\begin{eqnarray}\label{cottonweyl1}
	fC_{ijk}&=&R_{ijkl}\nabla^{l}f+\frac{R}{n-1}(\nabla_{i}fg_{jk}-\nabla_{j}fg_{ik})+\nabla_{j}fR_{ik}-\nabla_{i}fR_{jk}.
\end{eqnarray}
Note that 
\begin{eqnarray*}
   (n-2)\nabla_i(Sch)_{jk} = \nabla_{i}R_{jk}-\dfrac{1}{2(n-1)}\nabla_{i}Rg_{jk}.
\end{eqnarray*}
So,
\begin{eqnarray*}
   C_{ijk}= (n-2)[\nabla_i(Sch)_{jk}-\nabla_j(Sch)_{ik}]=(n-2)(Dsch)_{ijk}.
\end{eqnarray*}

Apply the Weyl tensor, i.e.,
\begin{eqnarray}\label{weyl}
	W_{ijkl}&=&R_{ijkl}-\frac{1}{n-2}\left(R_{ik}g_{jl}-R_{il}g_{jk}+R_{jl}g_{ik}-R_{jk}g_{il}\right)\nonumber\\
	&+&\frac{R}{(n-1)(n-2)}\left(g_{ik}g_{jl}-g_{il}g_{jk}\right),
\end{eqnarray} 
in equation \eqref{cottonweyl1} to obtain
\begin{eqnarray}\label{cottonweyl}	
	f(n-2)(DSch)_{ijk}&=&W_{ijkl}\nabla^{l}f+\frac{1}{n-2}\big(R_{jl}\nabla^{l}fg_{ik}-R_{il}\nabla^{l}fg_{jk}\big)\nonumber\\&+&\frac{R}{n-2}(\nabla_{i}fg_{jk}-\nabla_{j}fg_{ik})+\frac{n-1}{n-2}(\nabla_{j}fR_{ik}-\nabla_{i}fR_{jk}).
\end{eqnarray}	
For any tangent vector field $U$ orthogonal to $\nabla f$ we can apply the triple $(U,\,\nabla f,\,\nabla f)$ into the above equation, i.e.,
\begin{eqnarray*}	
	f(n-2)(DSch)(U,\,\nabla f,\,\nabla f)
	&=&W(U,\,\nabla f,\,\nabla f,\,\nabla f)+|\nabla f|^2Ric(U,\,\nabla f)\nonumber\\
	&=&|\nabla f|^2Ric(U,\,\nabla f).
\end{eqnarray*}	
Therefore, we have the result if the Schouten tensor is Codazzi, and this is true if the scalar curvature is constant and the curvature tensor is harmonic.

\iffalse
The Schouten tensor gives us
\begin{eqnarray*}
	(n-2)(DSch)(U,\,\nabla f,\,\nabla f) = (DRic)(U,\,\nabla f,\,\nabla f) - \frac{|\nabla f|^2}{2(n-1)}g(\nabla R,\,U).  
\end{eqnarray*}

\fi

Moreover, from \eqref{eqstfp} we have
\begin{eqnarray*}	
	0=fRic(U,\,\nabla f)=\frac{1}{2}g(\nabla|\nabla f|^2,\,U). 
\end{eqnarray*}	
Thus, $|\nabla f|$ must be constant at $f^{-1}(c)$. Taking into account Lemma \ref{lmasood_O} we have $$-fg(\nabla\rho,\,U) = (\mu+\rho)g(\nabla f,\,U)$$
i.e., $\rho$ is constant at $f^{-1}(c)$.
\end{proof}

%%%%%%%%%%%%%%%%%%%%%%%%%%%%%%%%%%%%%%%%%%%%%%%%%%%%%%%%%%%%%%%%%%%%%%%%%%%%%%%%%%%%%%%%%%%%%%%%%%%%%%%%%%%%%%%%%%%%%%%%%%%%%%%%%%%%%%%%%%%%%%%%%%%%%%%%%%%%%%%%%%%%%%%%%%%%%

\begin{lemma}\label{interesting1}
Consider a non-trivial static perfect fluid space with constant scalar curvature and Codazzi Ricci tensor. For each regular value $c$ of $f$ and a tangent vector $Y$ to $\Sigma=f^{-1}(c)$. Then, $Ric(N,\,N)$ is constant at $f^{-1}(c)$, where $N=-\dfrac{\nabla f}{|\nabla f|}$.
\end{lemma}
\begin{proof}
Let us prove that $Ric(N,\,N)$ is constant at $f^{-1}(c).$ To that end, we use \eqref{eqstfp} and Lemma \ref{interesting} to get
\begin{eqnarray*}
	\nabla_{N}N=\displaystyle\sum g(\nabla_{N}N,\,E_i)E_i &=& \textcolor{red}{-}\frac{1}{|\nabla f|}\displaystyle\sum g(\nabla_{N}\nabla f,\,E_i)E_i =\textcolor{red}{-} \frac{1}{|\nabla f|}\displaystyle\sum g(\nabla_{N}\nabla f,\,E_i)E_i\nonumber\\
	&=& f\displaystyle\sum Ric(N,\,E_i)E_i=0,
\end{eqnarray*}
where $\{E_i\}_{i=1}^{n-1}$ is any orthonormal frame for $T_{p}\Sigma,$ where $p\in\Sigma.$ Now, for any $Y\perp N$ we obtain
\begin{eqnarray*}
	Y\big(Ric(N,\,N)\big) &=& \nabla_{Y}Ric(N,\,N) + \underbrace{2Ric (\nabla_{Y}N,\,N)}_{=0}\underbrace{=}_{Codazzi} \nabla_{N}Ric(Y,\,N) \nonumber\\
	&=& N\big(\underbrace{Ric(Y,\,N)}_{=0}\big)-Ric(\nabla_NY,\,N) - Ric(Y,\,\underbrace{\nabla_NN}_{=0})\nonumber\\
	&=& - Ric(\nabla_{N}Y,\,N) = -Ric(N,\,N)g(\nabla_{N}Y,\,N)=0.
\end{eqnarray*}
In the above equality, we used that
\begin{eqnarray*}	Ric(\nabla_{N}Y)=g(\nabla_{N}Y,\,N)Ric(N)+\displaystyle\sum_{i}g(\nabla_{N}Y,\,E_i)Ric(E_i),
\end{eqnarray*}
and therefore
\begin{eqnarray*}	Ric(\nabla_{N}Y,\,N)=g(\nabla_{N}Y,\,N)Ric(N,\,N)+\displaystyle\sum_{i}g(\nabla_{N}Y,\,E_i)\underbrace{Ric(E_i,\,N)}_{=0}=0.
\end{eqnarray*}
Observe that, $g(\nabla_{N}Y,\,N)+g(Y,\,\nabla_NN)=N\big(g(Y,\,N)\big)=0$, hence $g(\nabla_{N}Y,\,N)=0.$
\end{proof}

	\section{The Level Sets of a Static Perfect Fluid Space }\label{proofspfcminihy}

 {In this section, we present the proofs for some of our results about static perfect fluid space and closed minimal hypersurface.}

		\begin{proof}
		[{\bf Proof of Theorem \ref{lemmaf}}]
		From the stability of $\Sigma$, for any $\phi\in C^1(\Sigma)$, we have
		\begin{eqnarray*}
			\int_\Sigma |\nabla_\Sigma \phi|^2 \, d\sigma\ge \int_\Sigma\left( |A_{\Sigma}|^2 + Ric(N, N)\right) \phi^2\, d\sigma \ge \int_\Sigma Ric(N, N) \phi^2\, d\sigma, 
		\end{eqnarray*}
		where $N$ is {an} unit normal vector field to $\Sigma$ and $d\sigma$ is the $(n-1)$-volume measure of hypersurfaces. {We infer that}
		\begin{eqnarray*}
			-\int_{\Sigma}\phi\Delta_{\Sigma}\phi d\sigma=\int_\Sigma |\nabla_\Sigma \phi|^2 \, d\sigma \ge \int_\Sigma\left( |A|^2 + Ric(N, N)\right) \phi^2\, d\sigma \ge \int_\Sigma Ric(N, N) \phi^2\, d\sigma.
		\end{eqnarray*}
		Thus,
		$$\int_\Sigma \phi \bigg[\bigg(\Delta_\Sigma + Ric(N, N)\bigg) \phi\bigg]\, d\sigma \le 0. $$
		
		It implies that the first eigenvalue, denoted by $\lambda$, of the operator $\Delta_\Sigma + Ric(N, N)$ is non-positive, where $\Delta_\Sigma$ is the induced Laplacian. Therefore,
		%\textcolor{red}{Nesta desigualdade $\phi\geq 0$, senão é falsa}
		\begin{equation}\label{operadorsinal}
			\bigg(\Delta_\Sigma + Ric(N, N)\bigg) \phi=\lambda \phi \le 0.
		\end{equation}

		On the other hand, since $H_\Sigma=0$, we have 
		\begin{equation}\label{laplacianos}
			\triangle f=\triangle_\Sigma f+\nabla^2f(N,N).
		\end{equation}
		
		Putting equations \eqref{eqstfp} and \eqref{eq2} in \eqref{laplacianos}, we obtain
		\begin{eqnarray}\label{uhum}
			\Delta_{\Sigma}f+Ric(N,N)f&=&\Delta f - \nabla^2f(N,N)+Ric(N,N)f\nonumber\\
			&=& \bigg(\frac{n-2}{(n-1)}\mu+\frac{n}{n-1}\rho\bigg)f+\left(\dfrac{\mu-\rho}{n-1}\right)f  \nonumber\\
			&=& (\mu+\rho)f.
		\end{eqnarray}
		
		Thus, from \eqref{operadorsinal} and WEC ($\mu+\rho\geq0$), then

		\begin{itemize}
			\item[(i)] Either $f$ is zero on $\Sigma$ (i.e., $\{f=0\}\subseteq\Sigma$), 
			\item[(ii)] or $f$ is the first eigenfunction of $\triangle_\Sigma +Ric(\nu,\nu)$ and $f\Big|_{\Sigma}\neq0$. Moreover, we must have $\lambda=\mu+\rho=0$ in $\Sigma$.
		\end{itemize}
		If the $(i)$ occurs, then $\Sigma$ is totally geodesic (see \cite{coutinho}). In fact, since $\nabla f$ does not vanish at $\{f=0\}$ (see, for instance, Lemma 1 in \cite{leandroernanipina}). Thus, $N=-\frac{\nabla f}{|\nabla f|}$ is the outward unit vector. Besides, it follows from \eqref{eqstfp} and \eqref{eq2} that $$\nabla^{2}f=\Delta f=0$$ in $\{f=0\}.$ We consider an orthonormal frame $\{e_1,\ldots,e_{n-1},e_n=-\frac{\nabla f}{|\nabla f|}\}$ on $\{f=0\}$, we have
		\begin{eqnarray*}
			X(|\nabla f|^2)&=&2\langle\nabla_X\nabla f,\nabla f\rangle\nonumber\\&=&2\nabla^2f(X,\nabla f)\nonumber\\&=&2fRic(X,\,\nabla f)-\frac{2(\mu-\rho)}{(n-1)}fg(X,\,\nabla f)=0,
		\end{eqnarray*}
		for any $X\in\mathfrak{X}(\{f=0\}),$ where in the second equality we use the equation \eqref{eqstfp}. Hence, $|\nabla f|$ is a non-null constant on $\{f=0\}.$ Therefore, using that $\nabla^2f=0$ in the second fundamental form, we obtain
		\begin{equation}\label{2formfund}
			A_{ab}=-\langle\nabla_{e_a}N,e_b\rangle=\frac{1}{|\nabla f|}\nabla_a\nabla_bf=0.
		\end{equation}
		This prove that $\Sigma\subseteq\{f=0\}$ is a totally geodesic. 
		
		Now, if $(ii)$ holds, from the stability inequality, we obtain
		\begin{eqnarray*}
			0\le\int_\Sigma |A|^2  f^2\, d\sigma&\le&  \int_\Sigma |\nabla_\Sigma f|^2 \, d\sigma-\int_\Sigma Ric(N,N) f^2 \, d\sigma\\
			&=&
			-\int_\Sigma f\triangle_\Sigma f \, d\sigma-\int_\Sigma Ric(N,N) f^2 \, d\sigma
			\\
			&=&-\int_\Sigma f(\triangle_{\Sigma}f+ Ric(N,N)f) \, d\sigma=0
			%\ge & \ge \int_\Sigma |A|^2  f^2\, d\sigma \\
			%&\Rightarrow & -\int_\Sigma f\triangle_\Sigma f \, d\sigma-\int_\Sigma Ric(\nu,\nu) %f^2 \, d\sigma \ge \int_\Sigma |A|^2  f^2\, d\sigma \\
			%&\Rightarrow & -\int_\Sigma f\underbrace{\bigg(\triangle_\Sigma f+Ric(\nu,\nu) f\bigg)}_{=0} \, d\sigma \ge \int_\Sigma |A|^2  f^2\, d\sigma \\
			%&\Rightarrow & 0 \ge \i4$nt_\Sigma |A|^2  f^2\, d\sigma \\
		\end{eqnarray*}
		Thus, $|A|=0$ and $\Sigma$ is also totally geodesic.
		
		Finally, we admit that second statement holds and $R_{\Sigma}+\rho\geq 0.$ Since $\Sigma$ is totally geodesic and $\mu=R/2,$ then by Gauss equation, {(Gauss Equation: $R-2Ric(N,N)=R_{\Sigma}-H^2+|A|^2$),} we have $Ric(N,\,N)=\mu-R_{\Sigma}{/{2}}.$ Using the equation \eqref{uhum}, we obtain that $$f\Delta_{\Sigma}f =f^{2}\left(\rho+\dfrac{R_{\Sigma}}{2}\right),$$ this implies that 
		\begin{eqnarray*}
\int_{\Sigma}\dfrac{|\nabla_{\Sigma}f|^{2}}{f^2}=\int_{\Sigma}\left(\rho+\frac{R_{\Sigma}}{{{2}}}\right)=\int_{\Sigma}\rho + 2\pi\mathfrak{X}(\Sigma).
		\end{eqnarray*}
  Here, $\mathfrak{X}(\Sigma)$ stands for the Euler number of $\Sigma$. We use the Gauss-bonnet theorem in the last identity.
  
		We conclude that 
\begin{eqnarray*}
\int_{\Sigma}\rho + 2\pi\mathfrak{X}(\Sigma)=0.
		\end{eqnarray*}
 if and only if $f$ is constant on $\Sigma.$
	\end{proof}

	The next result is based on an argument presented in \cite[Proposition 3.2]{Brendle2013} which we extended the ideas to an $n$-dimensional $(M^{n},\,g,\,f)$, $n\ge 3$,  solution for the static perfect fluid equations satisfying the NEC. Let $\Sigma$ be a two-sided smooth hypersurface in $U\subseteq M^n$ in which $f>0$. Let $N$ denote the outward-pointing unit normal to $\Sigma$. Consider the conformally modified metric $\bar{g}=\frac{1}{f^2} g$. Furthermore, we can infer the existence of $\varepsilon>0$ in which $\Phi: \Sigma \times [0,\varepsilon) \rightarrow U$, the normal exponential map with respect to the metric $\bar{g}$, is well-defined and such that $f>0$ in $\Sigma_t = \Phi(\Sigma\times\{t\})$. Precisely, for each point $x\in \Sigma$ the curve $\gamma_x(t)=\Phi(x,t)$ is a geodesic with respect to  $\bar{g}$ such that
	\begin{equation*}
		\Phi(x,0)=x \mbox{ and } \Phi(x,t)=exp_x(-tf(\gamma(t))N).
	\end{equation*}
	Moreover,
	\begin{eqnarray}\label{expflow}
		\frac{\partial \Phi}{\partial t}(x,t)\bigg|_{t=0}=-f(x)N(x).
	\end{eqnarray}
	Therefore, $\Sigma_t$ is the hypersurface obtained by pushing out along the normal geodesics to $\Sigma$ in the metric $\bar{g}$ a signed distance $t$. Note that the geodesic $\gamma$ has unit speed with respect to $\bar{g}$. Let 
	$H(x, t)$ and $A(x, t)$ be the mean curvature and second fundamental form of $x\in\Sigma_t$ with
	respect to $N$ in the metric $g$ (see also \cite[pg. 60]{galloway1993}). It is well known that \eqref{expflow} is smooth depending on the injectivity radius of $(U\backslash\partial U,\,\bar{g})$. Here, we will assume that the above flow is smooth. The ideas to prove the following lemma can be found in \cite{ambrozio2017,Brendle2013,galloway1993,huang2018}.
	
	\begin{lemma} \label{lemma:monotonicity}
		Let $(M^{n},\,g)$, $n\ge 3$, be an $n$-dimensional solution for the static perfect fluid equations. The mean curvature $H$ and second fundamental form $A$ of $\Sigma_t$ satisfies 
		\begin{equation*}
			\frac{d}{dt}\left(\frac{H}{f} \right) = |A|^2  + (\mu+\rho).   
		\end{equation*}
	\end{lemma}
	Considering the NEC and $f>0$ on  $\Sigma_0=\Sigma$ (locally area minimizing), then $\Sigma_t$ is totally geodesic for all $t\in[0,\,\varepsilon).$
	\begin{proof}
		Since  $\Delta_{\Sigma_t} f=\Delta f-\nabla^2 f(N,N)-H\langle \nabla f,N\rangle$, we have from \eqref{eqstfp} and \eqref{eq2} that 
		\begin{eqnarray}\label{eq001}
			\Delta_{\Sigma_t}f+fRic(N,N)&=&\Delta f-\nabla^2f(N,N)-H\langle \nabla f,N\rangle+fRic(N,N) \nonumber\\
			&=&\bigg(\frac{n-2}{n-1}\mu+\frac{n}{n-1}\rho\bigg)f+\frac{(\mu-\rho)}{(n-1)}f-H\langle \nabla f,N\rangle\nonumber\\
			&=&(\mu+\rho)f -H\langle \nabla f,N\rangle.
			%&=&\underbrace{(\mu+\rho)f}_{\ge 0}-H\langle \nabla f,\nu\rangle\ge -H\langle \nabla f,\nu\rangle.
		\end{eqnarray}
	
		It is well known that the mean curvature of $\Sigma_t$ satisfies the evolution equation (see \cite[Lemma 7.6]{bethuel2006})
		\begin{equation}\label{dercurvmed}
			\frac{\partial H}{\partial t}=f|A|^2+\Delta_{\Sigma_t}f+fRic(N,N).
		\end{equation}
		From \eqref{eq001} and the above equation, we conclude that
		\begin{equation}\label{dercurvmed1}
			\frac{\partial H}{\partial t} = f|A|^2+(\mu+\rho)f -H\langle \nabla f,N\rangle.\nonumber
		\end{equation}
		
		Moreover, from \eqref{expflow}, we have 
		\begin{equation*}
			\frac{\partial f}{\partial t}=d(f\circ\Phi)(\partial_t)=df\left(\frac{\partial\Phi}{\partial t}\right)=\Big\langle\nabla f,\,\frac{\partial\Phi}{\partial t}\Big\rangle=-f\langle \nabla f,N \rangle.
		\end{equation*}
		This implies that
		\begin{eqnarray*}
			\frac{d}{dt}\bigg(\frac{H}{f}\bigg)&=&-\frac{H}{f^2}\frac{\partial f}{\partial t}+\frac{1}{f}\frac{\partial H}{\partial t}\\
			&=&-\frac{H}{f^2}\frac{\partial f}{\partial t}+\frac{1}{f}\bigg[f|A|^2+(\mu+\rho)f -H\langle \nabla f,N\rangle\bigg]\\
			&=&\frac{H}{f}\langle \nabla f,N \rangle+|A|^2-\frac{H}{f} \langle \nabla f, N \rangle 
			+(\mu+\rho)=|A|^2+(\mu+\rho).
		\end{eqnarray*}
		Thus,
		\begin{eqnarray}\label{eq003}
			\dfrac{d}{dt}\bigg(\dfrac{H}{f}\bigg) =  |A|^2 + (\mu+\rho).
		\end{eqnarray}
		Finally, we obtain
		\begin{equation*}
			\frac{H}{f}(x,t)-\frac{H}{f}(x,0)=\int_0^t (|A|^2+\mu+\rho)(x,\,s)ds.
		\end{equation*}
		
		So, if the initial hypersurface $\Sigma$ has nonnegative mean curvature $H(\,\cdot,\,0)$, we conclude that the hypersurface $\Sigma_t$ has nonnegative mean curvature for each $t$.
		
		Without loss of generality, assume $f>0$ on $\Sigma$. Consider the deformation $\Phi:\Sigma\times [0, \varepsilon) \to M$ given by the normal, exponential map with respect to the conformally modified metric $f^{-2} g$ in  a collar neighborhood of $\Sigma.$  Let $\Sigma_t = \Phi(\Sigma\times \{ t\})$ such that $\Sigma_0 = \Sigma$.  We consider $H(\cdot, t)$ and  $A(\cdot, t)$ as the mean curvature and the second fundamental form of $\Sigma_t$ in the metric $g$.  In the first part of this Lemma, we prove  that $H(\cdot, t)\ge 0$ for $t\in (0, \epsilon)$. Now from the first variation of area (see \cite{galloway1993}), we have
		\begin{equation*}
			|\Sigma_t| - |\Sigma_0| = \int_0^t\left( -\int_{\Sigma_s} fH(\cdot, s) \, d\sigma \right) ds,   
		\end{equation*}
		where $|\Sigma_t|$ is the area of $\Sigma_t, \ t\in [0,\varepsilon).$ Let $\varepsilon$ be a number sufficiently small. Since $\Sigma$ is locally area minimizing, the above identity implies that the mean curvature of $\Sigma_t$ cannot be strictly positive for $t<\varepsilon$. Hence $H(\cdot, t)\equiv 0$ and the $(n-1)$-volume of $\Sigma_t$ is a constant. Using the inequality \eqref{eq003}, we conclude that $A(\cdot, t) \equiv 0$ and  $\Sigma_t$ is totally geodesic for $t\in [0, \varepsilon)$ with respect to the metric $g$.
	\end{proof}
	
	{Now, we are ready to proof Theorem \ref{propositionsplitting}, applying Lemma \ref{lemma:monotonicity}.  }

 %%%%%%%%%%%%%%%%%%%%%%%%%%%%%%%%%%%%%%%%5%%%%%%%%%%%%%%%%%%%%%%%%%%%%%%%%%%%%%%%%%%%%%%%%%%%%%%%%%%
	
	\begin{proof}[\bf{Proof of Theorem \ref{propositionsplitting}}]
		First, we will show that $\mu+\rho=0$ in $\Sigma_t,\ \forall\, t\in [0,\varepsilon).$ Let $X, Y$ be vectors tangential to $\Sigma_t$. Then, using the first variation of the second fundamental form (see \cite[Lemma 7.6]{bethuel2006}), we obtain,
		\begin{equation}\label{eqtsplit}
			\nabla_{\Sigma_t}^2 f(X, Y) + Rm(N, X, Y, N) f=0,   
		\end{equation}
		where $\nabla_{\Sigma_t}$ denotes the connection  of $\Sigma_t$, $N$ is an unit normal vector to $\Sigma_t $ (both with respect to the metric $g$), and $Rm$ is the Riemann curvature tensor  of $(M, g)$ (with the sign convention that the Ricci tensor is the trace on the first and fourth components of $Rm$).  By Lemma \ref{lemma:monotonicity}, we know that $\Sigma_t$ is totally geodesic, this implies that $\nabla_{\Sigma_t}^2 f(X, Y) = \nabla^2 f(X, Y)$. With this last equation and \eqref{eqstfp}, we obtain 
		\begin{eqnarray*}
			%& &\nabla^2 f=fRic-\frac{(\mu-\rho)}{n-1}fg\\
			\nabla_{\Sigma_t}^2 f=fRic-\frac{(\mu-\rho)}{n-1}fg.
		\end{eqnarray*}
		%&\Rightarrow &- Rm(\nu, X, Y, \nu) f=fRic-\frac{(\mu-\rho)}{n-1}fg,\\
		%*}
	From \eqref{eqtsplit},  we have 
	\begin{equation*}
		fRm(N, X, Y, N)=-fRic(X, Y) +\frac{(\mu-\rho)}{n-1}fg(X,Y).
	\end{equation*} 
	We consider an orthonormal frame $\{ E_i \}$ on $\Sigma_t$, then 
	\begin{eqnarray*}
		Ric(X, Y)& =& Rm(N, X, Y, N) + \sum_i Rm(E_i, X, Y, E_i)\\
		&=& - Ric(X, Y) + \frac{(\mu-\rho)}{n-1}g(X,Y)+  Ric_{\Sigma_t} (X, Y),
	\end{eqnarray*}	
	where we also use the Gauss equation in the second equality of the last equation, and we denote by $Ric_{\Sigma_t}$ as the Ricci tensor of $\Sigma_t$ induced from $g$. It gives that, for all tangential vector fields $X, Y$ to $\Sigma_t$, 
	\begin{equation} \label{equation:Ricci}
		Ric(X, Y) = \frac{1}{2} Ric_{\Sigma_t}(X, Y)+\frac{(\mu-\rho)}{2(n-1)}g(X,Y)
	\end{equation}
	and hence, combining the previous formulas and \eqref{eqstfp}, we {infer that}
	$$\nabla^2 f=\frac{1}{2} fRic_{\Sigma_t}-\frac{(\mu-\rho)}{2(n-1)}fg{.}$$
	%\begin{eqnarray*}
	%& & Ric(X, Y) = \frac{1}{2} Ric_{\Sigma_t}(X, Y)+\frac{(\mu-\rho)}{2(n-1)}g(X,Y)\\
	%&\Rightarrow & \frac{\nabla^2 f}{f}+\frac{(\mu-\rho)}{(n-1)}g=\frac{1}{2} Ric_{\Sigma_t}+\frac{(\mu-\rho)}{2(n-1)}g 
	% \quad\Rightarrow\quad  \frac{\nabla^2 f}{f}=\frac{1}{2} Ric_{\Sigma_t}-\frac{(\mu-\rho)}{2(n-1)}g \\
	%  &\Rightarrow & \nabla^2 f=\frac{1}{2} fRic_{\Sigma_t}-\frac{(\mu-\rho)}{2(n-1)}fg.
	% \end{eqnarray*}
Since, $\nabla_{\Sigma_t}^2 f = \nabla^2 f$, we obtain
\begin{equation}\label{equation:induced-static-1}
	\nabla_{\Sigma_t}^2 f =  \frac{1}{2} fRic_{\Sigma_t} -\frac{(\mu-\rho)}{2(n-1)}fg_{\Sigma_t}.
\end{equation}
Thus, taking the trace, we observe that
\begin{equation}\label{equation:induced-static-2}
	\triangle_{\Sigma_t}f+\frac{(\mu-\rho)}{2}f= \frac{1}{2}fR_{\Sigma_t}.
\end{equation}

Now, using \eqref{eqstfp} and \eqref{eq2}, we obtain
\begin{eqnarray*}
    \frac{1}{2}fR_{\Sigma_t}+fRic(N,\,N)&=&\bigg(\frac{n-2}{(n-1)}\mu+\frac{n}{n-1}\rho\bigg)f+\frac{(\mu-\rho)}{n-1}f +\frac{(\mu-\rho)}{2}f\nonumber\\
    &=&\frac{3\mu}{2}f + \frac{\rho}{2}f.
\end{eqnarray*}
Using that $\Sigma_t$ is totally geodesic in the Gauss equation, i.e.,
\begin{eqnarray*}
    \mu=\frac{R}{2}=\frac{R^{\Sigma_t}}{2} +2Ric(N,\,N)
\end{eqnarray*}
and the fact that $f\neq0,$ we have
$$f\frac{3\mu}{2} +f\frac{\rho}{2}=f\mu.$$
Finally, since $f\neq0$ at $\Sigma$ we obtain 
\begin{eqnarray}\label{legal}
	\mu+\rho=0\quad\mbox{in}\quad\Sigma_t,\quad\forall\, t\in[0,\,\varepsilon).
\end{eqnarray}
Therefore, from Lemma \ref{modaso_O}, i.e.,
\begin{eqnarray*}
   - fg(\nabla \rho,\,X) = (\mu+\rho)g(\nabla f,\,X) = 0;\quad\forall\, X\in T\Sigma_t,
\end{eqnarray*}
the scalar curvature $R=2\mu$ is constant at $\Sigma_t$ (unless $f=0$ at $\Sigma_t$).
Considering $\mu\geq0$, $\rho\geq0$ and \eqref{legal}, we conclude that $\mu=\rho=0$ in $\Sigma_t$. 

Second, we claim that $f$ is constant on $\Sigma_t$ for each $t\in[0,\,\varepsilon).$ In fact, from \eqref{equation:induced-static-2}, we get
\begin{eqnarray}\label{rsigma1}
	\int_{\Sigma_t} R_{\Sigma_t} d\sigma=2\int_{\Sigma_t}\frac{\triangle_{\Sigma_t} f}{f}d\sigma=2\int_{\Sigma_t}f^{-2} |\nabla_{\Sigma_t} f|^2  d\sigma\geq0.
\end{eqnarray}

Remember the following identity $\mathrm{div}_{\Sigma_t} \left(\nabla_{\Sigma_t}^2 f \right)=\nabla_{\Sigma_t}(\Delta_{\Sigma_t} f) + Ric_{\Sigma_t}(\nabla_{\Sigma_t}f),$ where $\mathrm{div}$ denotes the divergence of a $(0,2)$-tensor. Thus, taking the divergence of \eqref{equation:induced-static-1} and using \eqref{equation:induced-static-2}, we obtain that
\begin{eqnarray*}
	%0&=&d(\Delta_{\Sigma_t} f)+ Ric_{\Sigma_t}(\nabla_{\Sigma_t} f) -\mathrm{div}_{\Sigma_t} \left(\nabla_{\Sigma_t}^2 f \right)\\
	%&=&\dfrac{1}{2}d(fR_{\Sigma_t}) + Ric_{\Sigma_t}(\nabla_{\Sigma_t}f) -\frac{1}{2}\mathrm{div}_{\Sigma_t} \left( fRic_{\Sigma_t} \right)\\
	%&= & \frac{1}{2}\left[dfR_{\Sigma_t}+fdR_{\Sigma_t}\right] +Ric_{\Sigma_t}(\nabla_{\Sigma_t}f) -\frac{1}{2} \left[Ric_{\Sigma_t}(\nabla_{\Sigma_t}f)+f\mathrm{div}_{\Sigma_t} Ric_{\Sigma_t}\right]\\
	\frac{1}{2}\nabla_{\Sigma_t}fR_{\Sigma_t}+\frac{1}{2}f\nabla_{\Sigma_t}R_{\Sigma_t} + \frac{1}{2}Ric_{\Sigma_t}( \nabla_{\Sigma_t} f)-\frac{1}{2}f\mathrm{div}_{\Sigma_t}Ric_{\Sigma_t} =0.
\end{eqnarray*}
By the twice contracted second Bianchi identity $(\mathrm{div}_{\Sigma_t}Ric_{\Sigma_t}=\frac{1}{2}\nabla_{\Sigma_t}R_{\Sigma_t}),$ we arrive at
\begin{eqnarray*}
	%0&=&d(\Delta_{\Sigma_t} f)+ Ric_{\Sigma_t}(\nabla_{\Sigma_t} f) -\mathrm{div}_{\Sigma_t} \left(\nabla_{\Sigma_t}^2 f \right)\\
	%&=&\dfrac{1}{2}d(fR_{\Sigma_t}) + Ric_{\Sigma_t}(\nabla_{\Sigma_t}f) -\frac{1}{2}\mathrm{div}_{\Sigma_t} \left( fRic_{\Sigma_t} \right)\\
	%&= & \frac{1}{2}\left[dfR_{\Sigma_t}+fdR_{\Sigma_t}\right] +Ric_{\Sigma_t}(\nabla_{\Sigma_t}f) -\frac{1}{2} \left[Ric_{\Sigma_t}(\nabla_{\Sigma_t}f)+f\mathrm{div}_{\Sigma_t} Ric_{\Sigma_t}\right]\\
	\frac{1}{2}\nabla_{\Sigma_t}fR_{\Sigma_t}+\frac{1}{2}f\nabla_{\Sigma_t}R_{\Sigma_t} + \frac{1}{2}Ric_{\Sigma_t}( \nabla_{\Sigma_t} f)-\frac{1}{4}f\nabla_{\Sigma_t}R_{\Sigma_t} =0.
\end{eqnarray*}
This combined with \eqref{equation:induced-static-1} yields
\begin{eqnarray*}
	\frac{1}{2}\nabla_{\Sigma_t}fR_{\Sigma_t}+\frac{1}{4}f\nabla_{\Sigma_t}R_{\Sigma_t}+\frac{\nabla_{\Sigma_t}^2f(\nabla_{\Sigma_t} f)}{f}=0,
\end{eqnarray*}
which can be rewritten as
$$\frac{1}{4f}\nabla_{\Sigma_t}[(f^2R_{\Sigma_t})+2|\nabla_{\Sigma_t} f|^2]=0.$$

%&=&\frac{1}{4f}[2fdfR_{\Sigma_t}+f^2dR_{\Sigma_t}+4\nabla_{\Sigma_t}^2 f\cdot\nabla_{\Sigma_t} f]=\frac{1}{4f}[d(f^2R_{\Sigma_t})+2d|\nabla_{\Sigma_t} f|^2]\\
%
%&=&\frac{1}{4f}d[(f^2R_{\Sigma_t})+2|\nabla_{\Sigma_t} f|^2].    

%\begin{eqnarray*}
%\nonumber\\

This implies that $R_{\Sigma_t} f^2 + 2|\nabla_{\Sigma_t} f|^2=c,$ where $c$ is a constant on each $\Sigma_t$. Now, we will prove that $c=0.$ In fact, from $R_{\Sigma_t} f^2 + 2|\nabla_{\Sigma_t} f|^2 =c$ and \eqref{equation:induced-static-2}, we immediately get
$$c=\frac{2 \Delta_{\Sigma_t} f}{f}f^2+ 2|\nabla_{\Sigma_t} f|^2.$$
Integrating it over $\Sigma_t,$ it is easy to check that $R_{\Sigma_t} f^2 + 2|\nabla_{\Sigma_t} f|^2 =0.$ This combined with
\eqref{rsigma1}, we conclude that $ R_{\Sigma_t} = 0$ and $f$ is constant on $\Sigma_t$ for each $t\in [0,\epsilon)$. 
%\begin{eqnarray*}
%& &R_{\Sigma_t} f^2 + 2|\nabla_{\Sigma_t} f|^2 =c \Longrightarrow \frac{2 \Delta_{\Sigma_t} f}{f}f^2+ 2|\nabla_{\Sigma_t} f|^2=c\nonumber\\
%&\Longrightarrow &2 f \Delta_{\Sigma_t} f+ 2|\nabla_{\Sigma_t} f|^2 =c\\ &\Longrightarrow& 2 \int_{\Sigma_t}f \Delta_{\Sigma_t} f d\sigma+ 2\int_{\Sigma_t}|\nabla_{\Sigma_t} f|^2 d\sigma=\int_{\Sigma_t}cd\sigma\\
%&\Longrightarrow& -2 \int_{\Sigma_t}|\nabla_{\Sigma_t}|^2 d\sigma+ 2\int_{\Sigma_t}|\nabla_{\Sigma_t} f|^2 d\sigma=c\int_{\Sigma_t}d\sigma\\
%&\Longrightarrow& 0=c\int_{\Sigma_t}d\sigma \quad\Longrightarrow\quad c=0.\\
%\end{eqnarray*}

Therefore, from \eqref{equation:Ricci} and \eqref{equation:induced-static-1} we have $Ric_{\Sigma_t}(X,\,Y)=Ric(X, Y)=0$ for any tangent vectors in $\Sigma_t$. Moreover, since $A(\cdot,\,t)=R_{\Sigma_t}=R=0$ in $\Sigma_t$ for all $t\in[0,\,\varepsilon)$, by Codazzi and Gauss equations we obtain, respectively, $Ric(X,\,N)=0$ and $Ric(N,\,N)=0$. Thus, $Ric=0$ in $U\subseteq M$. Consequently, from \eqref{eqstfp} and \eqref{eq2} we can infer that $\mu=0$ in $U.$

Furthermore, from Lemma \ref{lmasood_O} we have 
\begin{eqnarray}\label{eqfrho}
	f\nabla\mu=\nabla[f(\mu+\rho)].
\end{eqnarray}
Using that $\mu=0$ in $U$ and \eqref{eqfrho}, we obtain that
$f\rho=k$ in $U$, where $k$ is a constant. Thus, $f\rho=k$ in $\Sigma_t$. However, $0=\mu=\rho$ in $\Sigma_t,$ we get $k=0$ in $U$. On the other hand, $f$ is a non-null constant in $\Sigma_t$ implying that $\rho=0$ in $U.$ Therefore, from Definition \ref{def1} we have $\nabla^{2}f=0$, and so $f$ is an affine function. 
\end{proof}

\begin{proof}[\bf{Proof of Corollary} \ref{coro1}] 
Suppose by contradiction that there is a closed, locally area minimizing surface $\Sigma$ in $M$. Using the Theorem \ref{propositionsplitting}, we obtain that $g$ must be Ricci-flat in an open neighborhood of the minimal surface. Since $ f > 0 $ on $M$, and $f$ and $g$ {are} analytic on $M^3$, $(M^3,\,g)$ has vanishing Ricci curvature. But, in three dimensions, this implies that $(M^3,\,g)$ is isometric to the Euclidean space, which does not have closed minimal surfaces.
\end{proof}

\begin{proof}[{\bf Proof of Theorem \ref{topospf}}]
We proved that $\Sigma_t$ is totally geodesic for each $t$ (Lemma \ref{lemma:monotonicity}). Considering an orthonormal frame $\{V,\,U,\,N_t\}$ on $\Sigma_t$ we have
\begin{eqnarray*}
	\frac{1}{f}\nabla^{2}_{\Sigma_t}f(U,\,U)&=&\frac{1}{f}\nabla^{2}f(U,\,U) \nonumber\\
	&=& Ric(U,\,U) - \frac{\mu-\rho}{2} \nonumber\\
	&=& sec(U\wedge V) + sec(U\wedge N_t)- \frac{\mu-\rho}{2}\nonumber\\
	&=& K_t + R(N_t,\,U,\,U,\,N_t) - \mu \nonumber\\
	&=& K_t -\mu - \frac{1}{f}\nabla^{2}_{\Sigma_t}f(U,\,U),
\end{eqnarray*}
where we used that $\mu=-\rho$ at $\Sigma_t$; see \eqref{legal}. From the arbitrariness of $U\in T\Sigma_t$, we conclude that
\begin{eqnarray*}
	\nabla^{2}_{\Sigma_t}f=\frac{1}{2}(K_t - \mu) f g_{\Sigma_t}.
\end{eqnarray*}
Consequently, 
\begin{eqnarray}\label{laplasigmat}
	\Delta_{\Sigma_t}f=(K_t - \mu)f
\end{eqnarray}
and
\begin{eqnarray*}
	\nabla^{2}_{\Sigma_t}f=\frac{1}{2}\Delta_{\Sigma_t}f g_{\Sigma_t}.
\end{eqnarray*}
Taking the divergence of the Hessian and using $\mathrm{div}_{\Sigma_t} \left(\nabla_{\Sigma_t}^2 f \right)=\nabla_{\Sigma_t}(\Delta_{\Sigma_t} f) + K_t\nabla_{\Sigma_t}f,$ we get
\begin{eqnarray*}
	0&=&\nabla_{\Sigma_t}(\Delta_{\Sigma_t} f) + 2K_t\nabla_{\Sigma_t}f\nonumber\\
	&=&\nabla_{\Sigma_t}[(K_t-\mu)f] + 2K_t\nabla_{\Sigma_t}f\\
	&=&f\nabla_{\Sigma_t}K_t +3K_t\nabla_{\Sigma_t}f - \mu\nabla_{\Sigma_t}f,
\end{eqnarray*}
we also used the fact that $\mu$ is constant at $\Sigma_t$. By a straightforward computation, 
\begin{eqnarray*}
\nabla_{\Sigma_t}[f^3(K_t-\frac{\mu}{3})]=0.
\end{eqnarray*}
Hence, 
\begin{eqnarray}\label{curvaturaK}
	K_t=\frac{c}{f^3} + \frac{\mu}{3},
\end{eqnarray}
where $c\in\mathbb{R}.$ Till this point, the argument was the same used by Ambrozio \cite{ambrozio2017}.

On the other hand, the decomposition of the curvature operator \eqref{weyl} leads to
\begin{eqnarray*}\label{wt}
     R(X,\,Y,\,W,\,Z)&=&[Ric(X,\,W)g(Y,\,Z)-Ric(X,\,Z)g(Y,\,W)\nonumber\\
     &&+Ric(Y,\,Z)g(X,\,W)-Ric(Y,\,W)g(X,\,Z)]\\
     &&-\frac{R}{2}[g(X,\,W)g(Y,\,Z)-g(X,\,Z)g(Y,\,W)].\nonumber
\end{eqnarray*}
Hence, for any $X,\,Y\in T_p\Sigma_t$ and $N_t$ normal to $\Sigma_t$ we obtain 
\begin{eqnarray*}
	R(N_t,\,X,\,Y,\,N_t) = [\mu - Ric(N_t,\,N_t)]g(X,\,Y) - Ric(X,\,Y).
\end{eqnarray*}
Since $\Sigma_t$ is totally geodesic, from the Gauss equation, we get
\begin{eqnarray*}
	Ric(X,\,Y)=R(N_t,\,X,\,Y,\,N_t) + Ric_{\Sigma_t}(X,\,Y).
\end{eqnarray*}
So,
\begin{eqnarray*}
	Ric(X,\,Y) = \frac{1}{2}Ric_{\Sigma_t}(X,\,Y) +\frac{1}{2}[\mu-Ric(N_t,\,N_t)]g(X,\,Y).
\end{eqnarray*}
Combine the above identity with $\mu+\rho=0$ and \eqref{equation:Ricci}, i.e.,
	\begin{equation*} 
		Ric(X, Y) = \frac{1}{2} Ric_{\Sigma_t}(X, Y)+\frac{\mu}{2}g(X,Y),
	\end{equation*}
 to see that
\begin{eqnarray*}
	Ric(N_t,\,N_t)=0.
\end{eqnarray*}

Therefore, from \eqref{eq001} we get $\Delta_{\Sigma_t} f=0$. Thus, 
\begin{eqnarray*}
	0=\Delta_{\Sigma_t}f=(K_t - \mu)f.
\end{eqnarray*}
{Finally, by \eqref{curvaturaK} we conclude that}
\begin{eqnarray*}
    K_t=\mu\quad\mbox{and}\quad f^3=\frac{3c}{2\mu}
\end{eqnarray*}
 are constant at $\Sigma_t$, i.e., $\mu(x,\,t)$ and $f(x,\,t)$ do not depends on $x\in M$.

 From the Gauss-Bonnet theorem, we can infer that
\begin{eqnarray*}
\mu|\Sigma_t|+\int_{\partial\Sigma_t}K^{g}da=2\pi\mathfrak{X}(\Sigma_t). 
\end{eqnarray*}
\end{proof}

\begin{proof}[{\bf Proof of Theorem \ref{cruztheorem}}]
   We start the proof by remembering the Gover-Orsted integral formula; see \cite[Theorem 20]{tiarlos}. If $T$ is a symmetric $2$-tensor field such that $\textnormal{div}T=0,$ and $X$ is a vector field on $M,$ then
    \begin{eqnarray*}
        \int_{M}X(\tau)=\frac{n}{2}\int_M\langle \mathring{T},\,\mathfrak{L}_{X}g\rangle - n \int_{\partial M} \mathring{T}(X,\,N),
    \end{eqnarray*}
    where $\tau=\textnormal{trace}(T),$ $\mathring{T}$ and $N$ stand for the trace free part of $T$ and the outer unit normal along $\partial M.$ Consider $T$ as the Einstein tensor, i.e., $T=Ric-\frac{R}{2}g$, and $\partial M=f^{-1}(c).$ Moreover, take $X=\nabla f$ and $N=-\dfrac{\nabla f}{|\nabla f|}.$ A straightforward computation gives us
    \begin{eqnarray*}
    \tau={\textnormal{trace}}(T)=R-\frac{n}{2}R=-\frac{(n-2)}{2}R\quad\mbox{and}\quad \mathring{T}=\mathring{R}ic=Ric-\frac{R}{n}g.
    \end{eqnarray*}

Assuming $X=\nabla f$, for any vector fields $Y,\,Z$ we have
\begin{eqnarray*}
    \left(\mathfrak{L}_{\nabla f}\right)g(Y,\,Z) = 2\nabla^2f(Y,\,Z).
\end{eqnarray*}
    Therefore, 
    \begin{eqnarray}\label{integral}
        -\frac{(n-2)}{2}\int_{M}\langle\nabla R,\,\nabla f\rangle = n\int_M \langle \mathring{R}ic,\,\nabla^2f\rangle - n \int_{\partial M} \mathring{R}ic(\nabla f,\,N).
    \end{eqnarray}
Moreover, by integration we have
    \begin{eqnarray*}
      -\frac{n-2}{2}\int_{M}\langle\nabla R,\,\nabla f\rangle = \frac{n-2}{2}\int_{M}R\Delta f + \frac{n-2}{2}\kappa\int_{\partial M}R  
    \end{eqnarray*}

    Now, we will analyze every term in the above identity separately. To start, considering that $(M^n,\,g)$ is a static perfect fluid we have
    \begin{eqnarray*}
      \langle \mathring{R}ic,\,\nabla^2f\rangle = \langle \mathring{R}ic,\,\mathring{\nabla}^2f\rangle  = f|\mathring{R}ic|^2.
    \end{eqnarray*}

By hypothesis, $\kappa=|\nabla f|$ is constant at $\partial M=f^{-1}(c)$. On the other hand, in case $c=0$  from \eqref{eqstfp} it is easy to see that
\begin{eqnarray*}
    fRic(\nabla f)=\frac{1}{2}\nabla|\nabla f|^2 +\frac{(\mu-\rho)}{(n-1)}f\nabla f.
\end{eqnarray*}
So, $\kappa=|\nabla f|$ is constant at $\partial M=f^{-1}(0).$ We conclude that in case $c=0$ the hypothesis over $\kappa$ is trivial.

{From the Gauss equation, i.e., 
\begin{eqnarray*}
    R=R^{\partial M} +2Ric(N,\,N) + |A|^2-H^2.
\end{eqnarray*}
and since $(n-1)|A|^2\geq H^2$, we obtain
\begin{eqnarray*}
    \frac{(n-2)}{2n}R+\dfrac{(n-2)}{2(n-1)}H^2-\frac{1}{2}R^{\partial M} \geq\mathring{R}ic(N,\,N),
\end{eqnarray*}
Considering $\mu=-\rho$ with $\mu$ constant, we get
  \begin{eqnarray}\label{cacons}
        2\pi\kappa\mathfrak{X}(\partial M) \leq \int_M f|\mathring{R}ic|^2 +\frac{\kappa}{3}\mu|\partial M|+\dfrac{\kappa}{4}\int_{\partial M}H^2.
    \end{eqnarray}

In the general case, considering $\partial M=f^{-1}(c)$ {and from equation \eqref{integral}} we can infer that
  \begin{eqnarray*}
       \frac{n-2}{2}\int_{M}R\Delta f + \frac{n-2}{2}\kappa\int_{\partial M}R  \leq n\int_M f|\mathring{R}ic|^2 + n \kappa\int_{\partial M} \left[\frac{(n-2)}{2n}R+\dfrac{(n-2)}{2(n-1)}H^2-\frac{1}{2}R^{\partial M} \right].
    \end{eqnarray*}}
    
Now, {we conclude that} 
  \begin{eqnarray*}
        \frac{n-2}{2n}\int_{M}R\Delta f+\kappa\int_{\partial M}\dfrac{R^{\partial M}}{2} \leq \int_M f|\mathring{R}ic|^2 +\dfrac{\kappa(n-2)}{2(n-1)}\int_{\partial M}H^2.
    \end{eqnarray*}
Moreover, if $n=3$ by Gauss-Bonnet we get
  \begin{eqnarray*}
       \frac{1}{6}\int_{M}f\mu(\mu+3\rho)+ 2\pi\kappa\mathfrak{X}(\partial M) \leq \int_M f|\mathring{R}ic|^2 +\dfrac{\kappa}{4}\int_{\partial M}H^2,
    \end{eqnarray*}
where $\mu=R/2$. Moreover, equality holds if $\partial M$ is totally umbilical. Let us combine the above inequality with the Hawking mass, i.e., 

  \begin{eqnarray*}
       \frac{1}{6}\int_{M}f\mu(\mu+3\rho)+ 2\pi\kappa\mathfrak{X}(\partial M) \leq \int_M f|\mathring{R}ic|^2 +4\pi\kappa\left(1 -\sqrt{\dfrac{16\pi}{|\partial M|}}\mathfrak{M}_{Haw}(\partial M)\right).
    \end{eqnarray*}
Hence,

  \begin{eqnarray*}
   \sqrt{\dfrac{16\pi}{|\partial M|}}\mathfrak{M}_{Haw}(\partial M)      \leq \int_M f\left[|\mathring{R}ic|^2 - \frac{1}{6}\mu(\mu+3\rho)\right] +4\pi\kappa - 2\pi\kappa\mathfrak{X}(\partial M).
    \end{eqnarray*}
equality holds if and only if $\partial M$ is totally umbilical.

\iffalse
Then, from \eqref{haw} and the above identity, we get
      \begin{eqnarray*}
    \mathfrak{M}_{Haw}(\partial M)\geq\sqrt{\dfrac{\partial M}{16\pi}}\left(\mu\dfrac{|\partial M|}{8\pi}+2-\mathfrak{X}(\partial M)\right),
    \end{eqnarray*}
    where the equality holds if $(M^3\,g)$ is an Einstein manifold. 
 \fi   

At the same time, by using once more \eqref{eq2} and that $f$ vanishes on $\partial M,$ we get $\Delta f=0$ over $\partial M.$ Therefore, from the formula for the second fundamental form, we deduce
\begin{equation}
A_{ab}=-\langle\nabla_{e_a}\nu,e_b\rangle=\frac{1}{|\nabla f|}\nabla_a\nabla_bf=\frac{\Delta f}{n|\nabla f|}g_{ab}=0,\nonumber
\end{equation} and hence, $\partial M$ is also a minimal hypersurface. Then, from \eqref{cacons} we have
  \begin{eqnarray*}
        2\pi\kappa\mathfrak{X}(\partial M) = \int_M f|\mathring{R}ic|^2 +\frac{\kappa}{3}\mu|\partial M|.
    \end{eqnarray*}

\end{proof}

Nonetheless, the technique employed in Theorem \ref{cruztheorem} is not directly applicable to studying perfect fluid in a broad and general form. The reason is that \eqref{eq2} is not as efficient as the Laplacian present in the vacuum structure. Inspired by Theorem 22 in \cite{tiarlos}, we present a topological classification for a surface ${\Sigma}=f^{-1}(c)$  in a static perfect fluid space without restricting $\mu$ and $\rho$. The following result also holds for the compact and non-compact cases of $(M^3,\,g)$.

	\begin{theorem}\label{cruztheorem007}
		Let $(M^3,\,g,\,f,\,\rho)$ be a three-dimensional (compact or non-compact)  static perfect fluid space with $\Sigma=f^{-1}(c)$. Then,
		\begin{eqnarray*}
			2\pi\mathfrak{X}(\Sigma)+\rho_{0}|\Sigma|\leq(4c)^{-1}\int_{\Sigma}fH\left(H-\dfrac{4|\nabla f|}{f}\right),
		\end{eqnarray*}
	where $\rho_0=\rho\Big|_{\Sigma}$ is a constant and $|\Sigma|$ is the area of $\Sigma$. The equality holds if $\Sigma$ is totally umbilical. In that case, if either
		\begin{eqnarray*}
			H < I^-
			\quad\mbox{or}\quad
			I^+ < H,
		\end{eqnarray*}
		where $H$ is the mean curvature of $\Sigma$, $\kappa=|\nabla f|$ at $\Sigma$ and {$I^{\pm}=\dfrac{2}{c}\left(\kappa \pm \sqrt{\kappa^2+c^2\rho_0}\right)$} 
	 the surface must be diffeomorphic to a sphere. Moreover, if $H\in[I^-,\,I^+]$, we can conclude that a photon surface must be diffeomorphic to a torus. 
	\end{theorem}

\begin{proof}[{\bf Proof of Theorem \ref{cruztheorem007}}]

   For a level set $\Sigma=f^{-1}(c)$ of $M$ we have
\begin{eqnarray*}
\Delta^{\Sigma}f+\nabla^{2}f(N,\,N)=\Delta f-H\langle\nabla f,\,N\rangle,
\end{eqnarray*}
we can infer that
\begin{eqnarray*}
\Delta^{\Sigma}f+f\mathring{R}ic(N,\,N)+H\langle\nabla f,\,N\rangle=\dfrac{(n-1)}{n}\Delta f.
\end{eqnarray*}
Here, $N=-\dfrac{\nabla f}{|\nabla f|}.$

{Now, using the Gauss equation $  R=R^{\Sigma} +2Ric(N,\,N) + |A|^2-H^2$ and the inequality $(n - 1)|A|^2 \geq H^2\nonumber$, we obtain }
\begin{eqnarray}\label{GCeq}
    \frac{(n-2)}{2n}R+\dfrac{(n-2)}{2(n-1)}H^2-\frac{1}{2}R^{\Sigma} \geq \mathring{R}ic(N,\,N),\nonumber
\end{eqnarray}
where the equality holds if $\Sigma$ is totally umbilical. Hence, 
\begin{eqnarray*}
\Delta f\leq\dfrac{n}{(n-1)}\Delta^{\Sigma}f+f\left[\frac{(n-2)}{2(n-1)}R+\dfrac{n(n-2)}{2(n-1)^2}H^2-\frac{n}{2(n-1)}R^{\Sigma}\right]+\dfrac{n}{(n-1)}H\langle\nabla f,\,N\rangle.
\end{eqnarray*}
Combine the above inequality with \eqref{eq2} to obtain
\begin{eqnarray*}
f\rho\leq\Delta^{\Sigma}f+f\left[\dfrac{(n-2)}{2(n-1)}H^2-\frac{1}{2}R^{\Sigma}\right]+H\langle\nabla f,\,N\rangle.
\end{eqnarray*}
Thus,
\begin{eqnarray*}
\int_{\Sigma}\frac{1}{2}fR^{\Sigma}\leq\int_{\Sigma}\left[H\left(\dfrac{(n-2)}{2(n-1)}fH-|\nabla f|\right)-f\rho\right].
\end{eqnarray*}
Assuming $n=3$ and $\Sigma=f^{-1}(c)$, the Gauss-Bonnet theorem gives us
\begin{eqnarray*}
2c\pi\mathfrak{X}(\Sigma)\leq\int_{\Sigma}\left[H\left(\dfrac{1}{4}fH-|\nabla f|\right)-f\rho\right].
\end{eqnarray*}
From Lemma \ref{lmasood_O}, we can infer that $\rho=\rho_0$ is constant at $\Sigma$. Therefore, 
\begin{eqnarray}\label{interest007}
	c(2\pi\mathfrak{X}(\Sigma)+\rho_{0}|\Sigma|)\leq\int_{\Sigma}H\left(\dfrac{1}{4}fH-|\nabla f|\right),
\end{eqnarray}
where $|\Sigma|$ is the area of $\Sigma$.
The equality holds if $\Sigma$ is totally umbilical. In that case, we can define following polynomial function $F(H)=\frac{f}{4}H^2-|\nabla f|H -f\rho.$ Therefore, if either $H>I^+$ or $H<I^-$ we have $\Sigma$ diffeomorphic to a sphere. Here, $I^{\pm}=\dfrac{2}{c}\left(\kappa \pm\sqrt{\kappa^2+c^2\rho_0}\right)$ and $\kappa=|\nabla f|$ is constant at $\Sigma$.
\end{proof}

\begin{remark}
	It is straightforward to obtain a Minskowski-type inequality from the above theorem, i.e., an inequality involving the total mean curvature of a surface. Let $(M^3,\,g,\,f,\,\rho)$ be a three-dimensional static perfect fluid space with $\Sigma=f^{-1}(c)$ with non-negative mean curvature $H$. Then,
	\begin{eqnarray*}
		2\pi\mathfrak{X}(\partial M)+\rho_{0}|\partial M|\leq\dfrac{1}{4}\int_{\partial M}H^2.
	\end{eqnarray*}
\end{remark}

\begin{proof}[{\bf Proof of Theorem \ref{cruztheorem007-c1}.}]
 From Theorem \ref{cruztheorem007} we can infer that
    \begin{eqnarray*}
			2\pi\mathfrak{X}(\Sigma)+\rho_{0}|\Sigma|\leq\dfrac{1}{4c}\int_{\Sigma}fH\left(H-\dfrac{4|\nabla f|}{f}\right),
		\end{eqnarray*}
  Considering our hypothesis, we get
    \begin{eqnarray*}
			2\pi\mathfrak{X}(\Sigma)+\rho_{0}|\Sigma|\leq\dfrac{1}{4}\int_{\Sigma}H^2 - \frac{\kappa}{c}\int_{\Sigma}H.
		\end{eqnarray*}
  Then, by definition of the Hawking mass and the Brown-York, we have
  \begin{eqnarray*}
      \dfrac{1}{4}\int_{\Sigma}H^2 = 4\pi-4\pi\sqrt{\dfrac{16\pi}{|\Sigma|}}\mathfrak{M}_{Haw}(\Sigma)
  \end{eqnarray*}
  and
  \begin{eqnarray*}
      \int_{\Sigma}H = 8\pi \mathfrak{M}_{BY}(\Sigma) - \int_{\Sigma}H_{0},
  \end{eqnarray*}
  respectively. Here, $H_0$ is the mean curvature of the isometric embedding of $\Sigma$ in $\mathbb{R}^3.$ Combine the last three equations to get
   \begin{eqnarray*}
			2\pi\mathfrak{X}(\Sigma)+\rho_{0}|\Sigma|\leq4\pi\left(1-\sqrt{\dfrac{16\pi}{|\Sigma|}}\mathfrak{M}_{Haw}(\Sigma)\right) -  \frac{\kappa}{c}\left(8\pi \mathfrak{M}_{BY}(\Sigma) - \int_{\Sigma}H_{0}\right),
		\end{eqnarray*}
  where $\mathfrak{X}(\Sigma)$ stands for the Euler number of $\Sigma.$ Thus,
  \begin{eqnarray*}
		4\pi\sqrt{\dfrac{16\pi}{|\Sigma|}}\mathfrak{M}_{Haw}(\Sigma) + 8\pi\frac{\kappa}{c} \mathfrak{M}_{BY}(\Sigma)\leq 4\pi - 	2\pi\mathfrak{X}(\Sigma) - \rho_{0}|\Sigma|  + \frac{\kappa}{c} \int_{\Sigma}H_{0},
		\end{eqnarray*}
i.e.,
      \begin{eqnarray*}
			2\sqrt{\dfrac{16\pi}{|\Sigma|}}\mathfrak{M}_{Haw}(\Sigma) +   \frac{4\kappa}{c}\mathfrak{M}_{BY}(\Sigma)\leq 2-\mathfrak{X}(\Sigma) + \frac{\kappa}{2\pi c}\int_{\Sigma}H_{0} - \dfrac{\rho_{0}}{2\pi}|\Sigma|.
		\end{eqnarray*}

Any constant mean curvature surface in the Schwarzschild space must be a sphere (see \cite{Brendle2013}). Let us consider the Schwarzschild space $(\mathbb{R}\times\mathbb{S}^2,\,g),$ where $$g=f^{-2}dr^2+r^2g_{\mathbb{S}^2}$$ and $f(x) = \sqrt{1-2\frac{\mathfrak{M}}{|x|}}$ with $g_{\mathbb{S}^2}$ denoting the canonical metric on ${\mathbb{S}^2}$. Here, $|x|=\sqrt{x_1^2+x_2^2+x_3^2}$ and $\kappa=|\nabla f|$. The metric $g$ is defined $\{0<|x|<2\,\mathfrak{M}\}$ {and} $\mathfrak{M}$ stands for the ADM mass. The mean curvature of the sphere $\partial B_r=\{|x|=r\}$, where $r=\dfrac{2\mathfrak{M}}{(1-c^2)}$, when it is embedded in $\mathbb{R}^3$ is 
    \begin{eqnarray*}
        H_0 = \dfrac{2}{r}.
    \end{eqnarray*}
Note that,
\begin{eqnarray*}
    \mathfrak{M}_{Haw}&=&\sqrt{\dfrac{|\partial B_r|}{16\pi}}\left(1 - \dfrac{1}{16\pi}\int_{\partial B_r}H_{0}^2\right)=\sqrt{\dfrac{4\pi r^2}{16\pi}}\left(1 - \dfrac{1}{16\pi}\dfrac{4}{r^2}4\pi r^2\right)=0.
\end{eqnarray*}

Moreover, 
\begin{eqnarray*}
    H=\dfrac{2}{r}f(r)=\dfrac{2}{r}\left(1-\dfrac{2\mathfrak{M}}{r}\right)^{1/2},
\end{eqnarray*}
and
\begin{eqnarray*}
    \mathfrak{M}_{Haw}(\partial B_r)&=&\sqrt{\dfrac{|\partial B_r|}{16\pi}}\left(1 - \dfrac{1}{16\pi}\int_{\partial B_r}H^2\right)=\sqrt{\dfrac{4\pi r^2}{16\pi}}\left(1 - \dfrac{1}{16\pi} 4\left(\dfrac{r-2\mathfrak{M}}{r^3}\right) 4\pi r^2\right)=\mathfrak{M}.
\end{eqnarray*}
We also have the Brown-York mass given by 
\begin{eqnarray*}
    \mathfrak{M}_{BY}(\partial B_r)=\dfrac{1}{8\pi}\int_{\partial B_r}(H_0 - H) = \dfrac{1}{8\pi}\left(\dfrac{2}{r}-2\left(\dfrac{r-2\mathfrak{M}}{r^3}\right)^{1/2}\right)4\pi r^2 = r - \sqrt{r(r-2\mathfrak{M})}=\dfrac{2\mathfrak{M}}{1+c}.
\end{eqnarray*}
Note that if $c\to 1$, then $\mathfrak{M}_{BY}\to\mathfrak{M}.$

On the other hand, from Theorem \ref{cruztheorem007-c1} we have
	\begin{eqnarray*}
			\mathfrak{M}_{BY}(\partial B_r)&=& \dfrac{1}{8\pi}\int_{\partial B_r}H_0 - \dfrac{c}{2\kappa}\sqrt{\dfrac{16\pi}{|\partial B_r|}}\mathfrak{M}_{Haw}(\partial B_r)\\
   &=&\dfrac{1}{4\pi r}|\partial B_r| - \dfrac{c}{2\kappa}\sqrt{\dfrac{16\pi}{|\partial B_r|}}\mathfrak{M}_{Haw}(\partial B_r)\\
   &=& r - \dfrac{c}{r\kappa}\mathfrak{M} = \dfrac{2\mathfrak{M}}{(1-c^2)} - \dfrac{c(1-c^2)}{2\kappa}.
		\end{eqnarray*}
In the Schwarzschild space holds the identity
    \begin{eqnarray*}
          \kappa= \frac{(1-c^{2})^{2}}{4\mathfrak{M}}.
    \end{eqnarray*}
Hence, 
	\begin{eqnarray*}
			\mathfrak{M}_{BY}(\partial B_r)= 2\mathfrak{M}\left[\dfrac{1}{(1-c^2)} - \dfrac{c}{(1-c^2)}\right] = \dfrac{2\mathfrak{M}}{1+c}.
		\end{eqnarray*}
\end{proof}

\begin{theorem}
   Let $(M^3,\,g,\,f,\,\rho)$ be a conformally flat three-dimensional static perfect fluid space with $\Sigma=f^{-1}(c)$, where $c$ is a regular value of $f$. Then,
 \begin{eqnarray*}
2\pi\mathfrak{X}(\Sigma) =\left[H\left(\frac{1}{4}H-\frac{\kappa}{c}\right)-\rho_0\right]|\Sigma|.
\end{eqnarray*}
In particular, if $(M^3,\,g,\,f,\,\rho)$ is Einstein, we have
    \begin{eqnarray*}
2\pi\mathfrak{X}(\Sigma) = \left(\dfrac{1}{4}H^2+\dfrac{\mu}{3}\right)|\Sigma|. 
\end{eqnarray*}
\end{theorem}
\begin{proof}

From Proposition \ref{propbene}, for any locally conformally flat (or Einstein), static perfect fluid space, the set level set $f^{-1}(c)$ is totally umbilical, and the mean curvature of $\Sigma=f^{-1}(c)$ is constant. Moreover, from \eqref{meanlcf} we have
\begin{eqnarray*}
    H=-\dfrac{2}{\kappa}\left[\left(\lambda-\dfrac{R}{3}\right)f+\dfrac{\Delta f}{3}\right],
\end{eqnarray*}
where $\lambda$ is the eigenvalue of $Ric$ at $\Sigma$.

Thus, from \eqref{eq2} we {obtain}
 \begin{eqnarray*}
    H=\dfrac{2c}{\kappa}\left[\frac{1}{2}(\mu-\rho)-\lambda\right],
\end{eqnarray*}
and {using} \eqref{interest007}, we can infer that
\begin{eqnarray*}
2\pi\mathfrak{X}(\Sigma) = \int_{\Sigma}\left[\dfrac{1}{4}H^2+2\lambda-\mu\right].
\end{eqnarray*}
Considering $3\lambda=R$, i.e., $(M,\,g)$ an Einstein manifold, we have
\begin{eqnarray*}
2\pi\mathfrak{X}(\Sigma) = \int_{\Sigma}\left[\dfrac{1}{4}H^2+\dfrac{1}{3}\mu\right], 
\end{eqnarray*}
i.e., 
\begin{eqnarray*}
2\pi\mathfrak{X}(\Sigma) = \dfrac{1}{4}H^2|\Sigma|+\dfrac{1}{3}\int_{\Sigma}\mu. 
\end{eqnarray*}
On the other hand, if $2\lambda+Ric(N,\,N)=R${, } we get
\begin{eqnarray*}
2\pi\mathfrak{X}(\Sigma) = \int_{\Sigma}\left[\dfrac{1}{4}H^2-Ric(N,\,N)+\mu\right].
\end{eqnarray*}

	Furthermore, it is well-known that Proposition \ref{propbene} implies that $$(M^{3},\,g)=(I\times N^2,\, dr^2+\varphi^2\overline{g}),$$
 see the details in \cite{leandro2019}. The Ricci tensor of $(M^3,\,g)$ is
	\begin{eqnarray}\label{wps}
		R_{11}=-2\frac{\ddot{\varphi}}{\varphi},\qquad R_{1a}=0
	\end{eqnarray}
	and
	\begin{eqnarray*}
		R_{ab}=\overline{R}_{ab}-\left[(\dot{\varphi})^2+\varphi\ddot{\varphi}\right]\bar{g}_{ab}\qquad  (a,\,b\in\{2,\,3\}).
	\end{eqnarray*}
	Since $\overline{R}_{ab}=\frac{\overline{R}}{2}\overline{g}_{ab}$, {we obtain}
	\begin{eqnarray*}
		R_{ab}=\left[\frac{\overline{R}}{2}-(\dot{\varphi})^2-\varphi\ddot{\varphi}\right]\overline{g}_{ab}.
	\end{eqnarray*}
	On the other hand, {since 
		\begin{eqnarray*}
			R=\varphi^{-2}\overline{R}-2\left(\frac{\dot{\varphi}}{\varphi}\right)^2-4\frac{\ddot{\varphi}}{\varphi}, 
		\end{eqnarray*}
		we {infer}}
	\begin{eqnarray*}\label{Rbar=R}
		\overline{R} = \varphi^2R+2(\dot{\varphi})^2+4\varphi\ddot{\varphi}.
	\end{eqnarray*}

	\iffalse
	\begin{eqnarray*}\label{Rbar=R}
		\overline{R} = \varphi^2R+2\left(\frac{\varphi'}{\varphi}\right)^2+4\frac{\varphi''}{\varphi}.
	\end{eqnarray*}
	\fi

	Using that $R=2\mu$, we get
	\begin{eqnarray}\label{barR}
		\overline{R} = 2\varphi^2\mu+ 2(\dot{\varphi})^2 + 4\varphi\ddot{\varphi}.
	\end{eqnarray}
	Moreover, from \eqref{eqstfp} we know that
	$$\frac{1}{|\nabla f|^2}\langle\nabla|\nabla f|^2,\,\nabla f\rangle=f(2R_{11} + \rho - \dfrac{1}{2}R),$$ 
	and {by} \eqref{wps} we get
	\begin{eqnarray*}
		\langle\nabla|\nabla f|^2,\,\nabla f\rangle=f(\dot{f})^2\left[-4\frac{\ddot{\varphi}}{\varphi}+ \rho - \dfrac{1}{2}R\right].
	\end{eqnarray*}
	Hence, using that $\nabla f = \dot{f}\partial_r$ we obtain
	\begin{eqnarray*}
2\ddot{f}=f\left[-4\frac{\ddot{\varphi}}{\varphi}+ \rho - \dfrac{1}{2}R\right].
	\end{eqnarray*}
We can see that $R$ depends only on $r$, and for that reason, $R=2\mu$ must be constant at $\Sigma$. We can conclude that $\overline{R}$ does not depend on $\theta$. Therefore, $\overline{R}$ is a constant at $\Sigma$.

	%%%%%%%%%%%%%%%%%%%%%%%%%%%%%%%%%%%%%%%%%%%%%%%%%%%%%%%%%%%%%%%%%%%%%%%%%%%%%%%%%%%%%%%%%%%%%%%%%%%%%%%%%%%%%%%%%%%%%%%%%%%%%%%%%%%%%%%%%%%%%%%%%%%%%%%%%%%%%%%%%%
	\iffalse
	Moreover, we know that $\textnormal{div}(E)=0$ implies that $\langle\nabla\rho,\,\nabla f\rangle+\rho\Delta f=0$. So, {from \eqref{s1},} $f'\rho'+\rho(\rho^2(f')^2-\Lambda)f=0.$ Thus, we have
	\begin{eqnarray*}
		\rho\left(f''+2\frac{\varphi''}{\varphi}f\right)+\rho'f'=0.
	\end{eqnarray*}  
	\fi

	%%%%%%%%%%%%%%%%%%%%%%%%%%%%%%%%%%%%%%%%%%%%%%%%%%%%%%%%%%%%%%%%%%%%%%%%%%%%%%%%%%%%%%%%%%%%%%%%%
	
	\iffalse
	A straightforward computation proves (see \cite[Equation 2.1]{dominguez2018introduction}) that the mean curvature of $\Sigma$ also is given by
	\begin{eqnarray*}
		H = |\nabla f|^{-1}\left(\Delta f -  \frac{1}{2|\nabla f|^2}\langle\nabla|\nabla f|^2,\,\nabla f\rangle\right).
	\end{eqnarray*}
	Therefore, 
	\begin{eqnarray}\label{H1}
		H= \frac{1}{f'}\left[(\rho^2(f')^2-\Lambda)f- f''\right].
	\end{eqnarray}
	\fi
	
	%%%%%%%%%%%%%%%%%%%%%%%%%%%%%%%%%%%%%%%%%%%%%%%%%%%%%%%%%%%%%%%%%%%%%%%%%%%%%%%%%%%%%%%%%%%%%%%%%%%%%%%%%%%%%%%%%%%%%%%%%%%%%%%%%%%%%%%%%%%%%%%%%%%%%%%%%%%%%%%%%%%%%%%%%%%%%%%%%%%%%%%%%%%%%%%%%%%%%%
	
	Furthermore, from the static perfect fluid equation, $A_{ab}=-\dfrac{\nabla_a\nabla_bf}{|\nabla f|}$, and since $\Sigma$ is totally umbilical (see Proposition \ref{propbene}) we have
	\begin{eqnarray*}
		-\frac{1}{2}|\nabla f|Hg_{ab}=\nabla_a\nabla_bf=f\left[R_{ab}-\dfrac{1}{2}(\mu-\rho)g_{ab}\right].
	\end{eqnarray*}
	Thus,
	\begin{eqnarray*}
		\left(-\frac{1}{2}|\nabla f|H+\frac{(\mu-\rho)f}{2}\right)\varphi^2\overline{g}_{ab}=fR_{ab}.
	\end{eqnarray*}
	On the other hand, \begin{eqnarray*}
		fR_{ab}=f\left[\frac{\overline{R}}{2}-(\dot{\varphi})^2-\varphi\ddot{\varphi}\right]\overline{g}_{ab}.
	\end{eqnarray*}
	Then,
	\begin{eqnarray*}
		f\left[\mu + \dfrac{\ddot{\varphi}}{\varphi}\right]=\left(-\frac{1}{2}\dot{f}H+\frac{(\mu-\rho)f}{2}\right),
	\end{eqnarray*}
	i.e.,
	\begin{eqnarray*}
		H=-\dfrac{c}{\kappa}\left[(\mu+\rho) + 2\dfrac{\ddot{\varphi}}{\varphi}\right]=\dfrac{c}{\kappa}\left[R_{11}-(\mu+\rho)\right].
	\end{eqnarray*}
  Hence,
  	\begin{eqnarray*}
		-\rho-\dfrac{\kappa}{c}H = \mu-R_{11},
	\end{eqnarray*}
 and so

 \begin{eqnarray*}
2\pi\mathfrak{X}(\Sigma) = \int_{\Sigma}\left[\dfrac{1}{4}H^2-\rho-\dfrac{\kappa}{c}H\right]=\left[H\left(\frac{1}{4}H-\frac{\kappa}{c}\right)-\rho_0\right]|\Sigma|.
\end{eqnarray*}
\end{proof}

\section{Static Perfect Fluid and Free Boundary minimal Hypersurface }

Let $(M^n,g)$ be Riemannian manifold with smooth boundary $\partial M.$ We denote by $X$ the unit normal vector field along $\partial M$ that points outside $\partial M.$ Let $\Sigma^{n-1}$ be a submanifold of $M^n$ with boundary $\partial \Sigma.$ Suppose that $\Sigma$ is immersed in $M$ and that $\partial \Sigma$ is contained in $\partial M,$ i.e., $\Sigma \cap \partial M=\partial \Sigma.$ We consider $N$ a local unit vector field to $\Sigma$ and denote by $A_{\Sigma}$ the second fundamental form of $\Sigma,$ i.e., $A_{\Sigma}(Y,Z)=\langle \nabla_YN,Z\rangle,$ $Y,\ Z\in T_{p}\Sigma,\ p\in \Sigma,$ where $\nabla$ is the Levi-Civita connection of $M.$ Also, we denote by $\nu$ the outward pointing conormal along $\partial\Sigma$ in $\Sigma.$ We say that $\Sigma$ is free boundary if $\Sigma$ meets $\partial M$ ortogonally, this means that $\nu=X$ along $\partial \Sigma$.

Let $F: \Sigma\times (-\varepsilon,\varepsilon)\to\Sigma_t\subset M$ be a variation of $\Sigma=\Sigma_0$ such that for every $t\in(-\varepsilon,\varepsilon)$ the map $F_t:x\in\Sigma\longmapsto F(x,t)\in M$ is an immersion of $\Sigma$ in $M$ and $F_t(\partial \Sigma)$ is contained in $\partial M.$ It is well known that $\Sigma$ is a critical point for the first variation of the area that preserves the property $\Sigma\cap\partial M= \partial \Sigma$ if and only if $\Sigma$ is minimal, i.e., the mean curvature of $\Sigma$ denoted by $H_{\Sigma}=\text{trace}(A_{\Sigma})$ is zero with free boundary. Furthermore, $\Sigma$ is locally area-minimizing in $M$ if every nearby properly immersed hypersurface has an area greater than or equal to the area of $\Sigma.$ In particular, from the first variation of the area, we have that an area-minimizing properly immersed hypersurface $\Sigma$ is minimal and free boundary.

\begin{figure}[!ht]
\centering
\includegraphics[scale = 0.34]{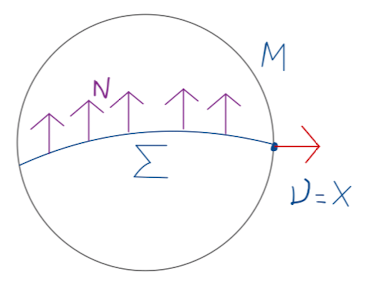}
\caption{The hypersurface $\Sigma$ is free boundary in $M$.}
\label{figfb}
\end{figure}

Moreover, if $\Sigma$ is minimal with a free boundary, then the second variation of the area of $\Sigma$ is given by
\begin{eqnarray}\label{eq000}
\frac{d^2}{dt^2}\Bigg|_{t=0}A(\Sigma_t)=-\int_{\Sigma}\phi\{\Delta_{\Sigma} \phi+(Ric(N,N)+|A_{\Sigma}|^2)\phi\}d\sigma+ \int_{\partial \Sigma}\phi\left\{\frac{\partial \phi}{\partial \nu}-II(N,N)\phi\right\}da,
\end{eqnarray}
where $\phi \in C^{\infty}(\Sigma),$ $ \Delta_{\Sigma}$ is the Laplace operator of $\Sigma$ with respect to the induced metric, $Ric$ is the Ricci tensor of $M,$ $\frac{\partial f}{\partial \nu}$ is the derivative of $f$ in the direction of the exterior normal $\nu$ and $II$ is the second fundamental form of $\partial M$ in $M$, i.e., $II(W,\,W)=\langle \nabla_{W}X,\,W\rangle,$ for all $W\in T_{p}\partial M,\ p\in\partial M.$ We say that $\Sigma$ is free boundary stable, if the second variation of area is nonnegative for every variation that preserves the boundary $\partial M.$ Finally, we say that $\Sigma$ is two-sided if there exists a unit vector field $N$ along $\Sigma$ that is normal to $\Sigma.$ For more details in this short introduction about free boundary hypersurface, see  \cite{ambrozio2015}.

%Moreover, we say that $(M^n,g)$ has mean convex boundary if the mean curvature in the boundary of $M$ denoted by $H^{\partial M}$ is nonnegative. 

%%%%%%%%%%%%%%%%%%%%%%%%%%%%%%%%%%%%%%%%%%%%%%%%%%%%%%%%%%%%%%%%%%%%%%%%%%%%%%%%%%%%%%%%%%%%%%%%%%%%%%%%%%%%%%%%%%%%%%%%%%%%%%%%%%%%%%%%%%%%%%%%%%%%%%%%%%%%%%%%%%%%%%%%%%%%%%%%%%%%%%%

We can settle the Jacobi operator 
\begin{eqnarray*}\label{jacobi}
L_{\Sigma}u= \Delta_{\Sigma}u + (Ric(N,\,N)  +|A_\Sigma|^2)u
\end{eqnarray*}
with boundary conditions
\begin{eqnarray*}
\frac{\partial u}{\partial\nu}-II(N,\,N)u=0\quad\mbox{on}\quad\partial\Sigma.
\end{eqnarray*}

We can characterize the solutions for the above problem as follows: Let $V_k$ denote the subspace spanned by the first $k$ eigenfunctions for the above system, then the value of the next eigenvalue $\lambda_{k+1}(L_\Sigma)$ equals the minimum of $Q$ on the $L^2(\Sigma,\,d\sigma)$ orthogonal complement of $V_k$, i.e., $$\lambda_{k+1}=\min_{u\in V^{\perp}_{k}\backslash\{0\}}\frac{Q(u,\,u)}{\int_\Sigma u^2da},$$
where 
\begin{eqnarray}\label{Q}
Q(u,\,u)=-\int_\Sigma u L_\Sigma u d\sigma+\int_{\partial\Sigma}u\left[\frac{\partial u}{\partial\nu} -II(N,N)u\right]da.
\end{eqnarray}

More precisely, if $\lambda_1<\lambda_2\leq\ldots\leq\lambda_k \ldots \to \infty$ is the list of eigenvalues of $L_\Sigma$ with boundary conditions inferred above (with repetitions) then the number of such negative eigenvalues coincides with the index of $\Sigma.$ A minimal surface with a free boundary is said to be stable if it
has nonnegative second variation $Q(u, u) \geq 0$, for all $u$.

\begin{proof}[{\bf Proof of Theorem \ref{propositionsplitting1}}]
By hypothesis $\Sigma$ is two-sided, a unit vector field $N$ along $\Sigma$ exists that is normal to $\Sigma$. We consider $X$ as the unit vector field on $\partial M$ that is normal to $\partial M$ and
points outside $M$. Since $\Sigma$ is free boundary, we obtain that the unit conormal $\nu$ of $\partial\Sigma$ that points outside $\Sigma$ coincides with $X$ along $\partial\Sigma$. Since $\Sigma$ is a minimal hypersurface, from \eqref{uhum} we have
$$L_{\Sigma}f=(\mu+\rho+|A_\Sigma|^2)f.$$ Therefore, since $\partial\Sigma=\Sigma\cap\partial M$ and $\partial M=\{f=c\}$ we can infer that $X=\frac{-\nabla f}{|\nabla f|}$ is the outward normal vector field. At this point, it is important to remember that Proposition \ref{propbene} holds true and we will use it from now on. Hence, from \eqref{Q} we obtain
\begin{eqnarray}\label{000000}
	Q(f,\,f) &=& -\int_\Sigma f L_\Sigma f+\int_{\partial\Sigma}f\frac{\partial f}{\partial\nu} -c^2\int_{\partial\Sigma}II(N,N)da\nonumber\\
	&=& -\int_\Sigma f L_\Sigma f-\int_{\partial\Sigma}f|\nabla f| -c^2\int_{\partial\Sigma}II(N,N)da\nonumber\\
	&=& -\int_\Sigma f L_\Sigma f-\kappa_1|\partial\Sigma| -c^2\int_{\partial\Sigma}II(N,N)da\nonumber\\
	&=& -\int_\Sigma f^2(\mu+\rho)-\int_\Sigma f^2|A_{\Sigma}|^2d\sigma -c^2\int_{\partial\Sigma}II(N,N)da - \kappa_1|\partial\Sigma|\nonumber\\
	&=& -\int_\Sigma f^2(\mu+\rho)-\int_\Sigma f^2|A_{\Sigma}|^2d\sigma -c^2\int_{\partial\Sigma}g(N,\,\nabla_{N}X)da -\kappa_1|\partial\Sigma|,
\end{eqnarray}
where $0<\kappa_1=f|\nabla f|$ at $\partial\Sigma.$ Remember, we are assuming $|\nabla f|$ constant in $\partial M = f^{-1}(c),$ and $\partial\Sigma\subset\partial M.$

% Hence, from \eqref{uhum} we get
%\begin{eqnarray*}
%Q(f,\,f) &=& -\int_\Sigma (f-c)(\mu+\rho)f-c\int_\Sigma Ric(N,\,N)f-\int_\Sigma |A_\Sigma|^2f^2-c^2\int_{\partial\Sigma}II(N,N)da\nonumber\\ 
%&=& -\int_\Sigma (f-c)(\mu+\rho)f-c\int_\Sigma Ric(N,\,N)f-\int_\Sigma |A_\Sigma|^2f^2-c^2\int_{\partial\Sigma}g(N,\,\nabla_{N}X)da\nonumber\\
%\end{eqnarray*}
%-----------------------------------------------------------------------------------------------------------------------------------------------------------------------------
\iffalse
Hence, \textcolor{red}{this line only if $f|_\Sigma=c$}
\begin{eqnarray*}
	Q(f,\,f) &=& -c^2\left[\int_\Sigma Ric(N,\,N)d\sigma+\int_\Sigma |A_\Sigma|^2d\sigma+\int_{\partial\Sigma}II(N,N)da\right]\nonumber\\
	&=& -c^2\left[\int_\Sigma Ric(N,\,N)d\sigma+\int_\Sigma |A_\Sigma|^2d\sigma+\int_{\partial\Sigma}g(N,\,\nabla_{N}X)da\right]\nonumber\\
\end{eqnarray*}
\fi
%---------------------------------------------------------------------------------------------------------------------------------------------------------------------------------
So, again using \eqref{eq2} and \eqref{uhum} we have
\begin{eqnarray*}
	g(N,\,\nabla_NX) = g(N,\,\nabla_N\left(\frac{-\nabla f}{|\nabla f|}\right)) 
	&=& \frac{-1}{|\nabla f|}g(N,\,\nabla_N\nabla f) = \frac{-1}{|\nabla f|}\nabla^2f(N,\,N)=\frac{-1}{|\nabla f|}(\Delta f-\Delta_{\Sigma}f)\nonumber\\
	&=& \frac{-f}{|\nabla f|}\left[\bigg(\frac{n-2}{(n-1)}\mu+\frac{n}{n-1}\rho\bigg)+ Ric(N,\,N)-(\mu+\rho)\right]\nonumber\\
	&=& \frac{-f}{|\nabla f|}\left[ Ric(N,\,N)-\frac{(\mu-\rho)}{n-1}\right].\nonumber\\
\end{eqnarray*}
Combining $Q(f,\,f)$ with the last identity gives us
\begin{eqnarray*}
	Q(f,\,f) = -\int_\Sigma f^2(\mu+\rho)-\int_\Sigma f^2|A_{\Sigma}|^2d\sigma - \kappa_1|\partial\Sigma| + \kappa_2\int_{\partial\Sigma}\left[ Ric(N,\,N)-\frac{(\mu-\rho)}{n-1}\right]da,\nonumber
\end{eqnarray*}
where $\kappa_2=\frac{f^3}{|\nabla f|}$ at $\partial\Sigma.$ From Lemma \ref{lmasood_O}, we have $\rho$ constant at $\partial M.$ By hypothesis, consider $\mu$ constant.

\iffalse
Considering $\mu+\rho=0$ (i.e., vacuum case), we get
\begin{eqnarray*}
	Q(f,\,f) = -\int_\Sigma f^2|A_{\Sigma}|^2d\sigma - \kappa_1|\partial\Sigma| + \kappa_2\int_{\partial\Sigma}\left[ Ric(N,\,N)-\frac{2\mu}{n-1}\right]da.\nonumber
\end{eqnarray*}
\fi

Applying the triple $(\nabla f,\,N,\,N)$ into \eqref{cottonweyl} and assuming zero radial Weyl curvature we get
\begin{eqnarray*}
	f{(n-2)}(DSch)(\nabla f,\,N,\,N)=\underbrace{W(\nabla f,\,N,\,N,\,\nabla f)}_{=0}-\frac{1}{n-2}Ric(X,\,X)+\frac{R}{n-2}-\frac{n-1}{n-2}Ric(N,\,N).
\end{eqnarray*}
The Schouten tensor gives us
\begin{eqnarray*}
	(n-2)(DSch)(\nabla f,\,N,\,N) = (DRic)(\nabla f,\,N,\,N) - \frac{1}{2(n-1)}g(\nabla R,\,\nabla f).  
\end{eqnarray*}
Thus, if $DRic=0$ and $\mu$ is constant we get $Dsch=0.$ Therefore, combining the last two identities we can infer that
\begin{eqnarray*}
	Ric(N,\,N)-\frac{\mu}{n-1}+\frac{\rho}{n-1}=\frac{\mu}{n-1}+\frac{\rho}{n-1}-\frac{1}{n-1}Ric(X,\,X).
\end{eqnarray*}
Therefore, from Lemma \ref{interesting1} we can infer that $Ric(N,\,N)-\frac{(\mu-\rho)}{n-1}$ is constant at $\partial\Sigma.$ So,
\begin{eqnarray*}
	Q(f,\,f) = -\int_\Sigma f^2|A_{\Sigma}|^2d\sigma - \kappa_1|\partial\Sigma| + \kappa_2\left[ Ric(N,\,N)-\frac{(\mu-\rho)}{n-1}\right]|\partial\Sigma|.\nonumber
\end{eqnarray*}

Considering the Gauss equation, we have
\begin{eqnarray*}\label{ollha}
	\frac{(n-2)}{(n-1)}\mu + \frac{\rho}{(n-1)} - \frac{R_\Sigma}{2}-\frac{|A_{\Sigma}|^2}{2} = Ric(N,\,N) - \frac{(\mu-\rho)}{n-1}. 
\end{eqnarray*}
In fact, from Proposition \ref{propbene} and the above equation we conclude that $K^{\Sigma}=\dfrac{R_\Sigma}{2}$ is constant at $f^{-1}(c).$

Then, 

\begin{eqnarray}\label{operatorforSPF}
	Q(f,\,f) = -\int_\Sigma f^2|A_{\Sigma}|^2d\sigma - \kappa_1|\partial\Sigma| + \kappa_2\left[\frac{(n-2)}{(n-1)}\mu+\frac{\rho}{(n-1)} - \frac{R_\Sigma}{2}-\frac{|A_{\Sigma}|^2}{2}\right]|\partial\Sigma|.
\end{eqnarray}
Taking $n=3$ in \eqref{operatorforSPF} and assuming $\Sigma$ stable we have
\begin{eqnarray*}
	0\leq Q(f,\,f) &=& -\int_\Sigma f^2|A_{\Sigma}|^2d\sigma - (\kappa_1+\kappa_2K^\Sigma)|\partial\Sigma| - \kappa_2\frac{|A_{\Sigma}|^2}{2}|\partial\Sigma|+\kappa_2\left[\frac{(n-2)}{(n-1)}\mu+\frac{\rho}{(n-1)}\right]|\partial\Sigma|\nonumber\\
 &=& -\int_\Sigma f^2|A_{\Sigma}|^2d\sigma - (\kappa_1+\kappa_2K^\Sigma)|\partial\Sigma| - \kappa_2\frac{|A_{\Sigma}|^2}{2}|\partial\Sigma|+\frac{\kappa_2}{2}\left[\mu+\rho\right]|\partial\Sigma|,
\end{eqnarray*}
where $K^\Sigma$ is the Gauss curvature of $\Sigma$. Thus, 
\begin{eqnarray*}
	\left[\kappa_1+\kappa_2K^\Sigma- \frac{\kappa_2}{2}\left(\mu+\rho\right)\right]|\partial\Sigma|&\leq&\int_\Sigma f^2|A_{\Sigma}|^2d\sigma + (\kappa_1+\kappa_2K^\Sigma)|\partial\Sigma| \\
 &&+ \kappa_2\frac{|A_{\Sigma}|^2}{2}|\partial\Sigma| - \frac{\kappa_2}{2}\left[\mu+\rho\right]|\partial\Sigma|\leq0.\nonumber
\end{eqnarray*}

Hence,
\begin{eqnarray*}
	\frac{1}{2}(\mu+\rho)-\frac{\kappa_1}{\kappa_2}\geq K^\Sigma\quad\mbox{at}\quad\partial\Sigma=\Sigma\cap\partial M.
\end{eqnarray*}
Moreover, $Q(f,\,f)=0$ if $\Sigma$ is totally geodesic and $\dfrac{1}{2}(\mu+\rho)-\dfrac{\kappa^2}{c^2} = K^\Sigma$. Here, $\kappa=|\nabla f|$ at $\partial\Sigma$. We can conclude that if $\mu+\rho=0$, then $K^{\Sigma}$ is nonpositive.

Let us consider the three-dimensional case with $\Sigma$ being compact. Remember that the mean curvature of $\partial M$ denoted by $H_{\partial M}$ is the trace of the shape operator. The
free boundary hypothesis implies that $K^{g}$, the geodesic curvature of $\partial\Sigma$ in $\Sigma$, can be
computed as $K^g = g(T,\,\nabla_{T}\nu) = g(T,\,\nabla_{T}X)$, where $T$ is a unit vector field tangent to $\partial\Sigma$. In particular,
$H_{\partial M}= K^{g} + g(N,\,\nabla_{N}X)=K^{g}+II(N,\,N)$, see \cite[Proposition 6]{ambrozio2015}. So, from \eqref{000000} we have
\begin{eqnarray*}
	Q(f,\,f) &=& -\int_\Sigma f^2(\mu+\rho)-\int_\Sigma f^2|A_{\Sigma}|^2d\sigma -c^2\int_{\partial\Sigma}[H_{\partial M} - K^{g}]da -\kappa_1|\partial\Sigma|\nonumber\\
	&=& -\int_\Sigma f^2(\mu+\rho)-\int_\Sigma f^2|A_{\Sigma}|^2d\sigma +c^2\int_{\partial\Sigma}K^{g}da -(c^2H_{\partial M}+\kappa_1)|\partial\Sigma|,
\end{eqnarray*}
where we used that $H_{\partial M}$ is constant (i.e., photon surface).
Applying the Gauss-Bonnet theorem, i.e., 
\begin{eqnarray*}
	\int_{\Sigma}K^\Sigma d\sigma+\int_{\partial\Sigma}K^{g}da=2\pi\mathfrak{X}(\Sigma). 
\end{eqnarray*}
Hence, 
\begin{eqnarray*}
	0\leq Q(f,\,f) =-\int_\Sigma f^2(\mu+\rho) -\int_\Sigma f^2|A_{\Sigma}|^2d\sigma -c^2\int_{\Sigma}K^\Sigma d\sigma +2c^2\pi\mathfrak{X}(\Sigma)-(c^2H_{\partial M}+\kappa_1)|\partial\Sigma|.
\end{eqnarray*}
Considering $\mu+\rho=0$ at $\Sigma$ we have $K^\Sigma\leq0$ at $\partial\Sigma$. Thus,
\begin{eqnarray*}
	2c^2\pi\mathfrak{X}(\Sigma)\geq(c^2H_{\partial M}+\kappa_1)|\partial\Sigma| + \int_\Sigma f^2|A_{\Sigma}|^2d\sigma + \int_\Sigma f^2(\mu+\rho) + c^2\int_{\Sigma}K^\Sigma d\sigma, 
\end{eqnarray*}
where $\kappa_1=f|\nabla f|\Big|_{\partial\Sigma}$ is constant. Considering $\kappa=|\nabla f|$ at $\partial\Sigma$, we obtain
\begin{eqnarray*}
	2\pi\mathfrak{X}(\Sigma)\geq \left(H_{\partial M}+\frac{\kappa}{c}\right)|\partial\Sigma|+ \int_{\Sigma}K^\Sigma d\sigma.
\end{eqnarray*}

At this point, we know that $K^{\Sigma}$ is constant (and non-positive) at $\partial\Sigma$. Since $\Sigma$ is compact, we can conclude that
$$0\geq K^{\Sigma}\geq  \min_{\Sigma}K^{\Sigma}.$$ Consider $\beta=|\min_{\Sigma}K^{\Sigma}|$ to obtain
\begin{eqnarray*}
	2\pi\mathfrak{X}(\Sigma)\geq \left(H_{\partial M}+\frac{\kappa}{c}\right)|\partial\Sigma| - \beta|\Sigma|.
\end{eqnarray*}

\end{proof}

\section{A static stellar model inspired by Witten's black hole}

	We will show a method inspired by \cite{barboza2018} to furnish spherical symmetric static perfect fluid solutions. In particular, we will provide an interesting new model inspired by Witten's black hole. The main contribution of this method is that it not requires any equation of state in the process like in the TOV solutions.
 
 The black hole model for the two-dimensional string theory proposed by Witten \cite{witten} is $\mathbb{R}\times\mathbb{S}$ with a metric tensor 
	\begin{eqnarray}\label{witten01}
	ds^{2}=dr^{2}+\tanh^{2}(r)d\theta^{2}.
	\end{eqnarray}
	 By making a proper choice of coordinates, we have another form to see the above metric: $$ds^{2}=\frac{1}{4r^{2}}\left(1-\frac{\mathfrak{M}}{r}\right)^{-1}dr^{2}+\left(1-\frac{\mathfrak{M}}{r}\right)d\theta^{2},$$
  where $\mathfrak{M}$ is a constant which represents the mass.

It is important to point out that in $r=0$ the metric \eqref{witten01} collapses into a point. This black hole model is also known as the Hamilton cigar, inspired by the two-dimensional steady Ricci soliton discovered by Richard Hamilton in its studies of Ricci flow. Our intention is to provide an $n$-dimensional static perfect fluid space model inspired by Witten's black hole.
 
 To that end, considering $(\mathbb{R}^{n}, g)$ to be the standard Euclidean space with metric $g$ and Cartesian coordinates $x=(x_{1}, \cdots, x_{n})$, with $g_{ij} = \delta_{ij}$, $1\leq i, j\leq n$, where $\delta_{ij}$ is the delta Kronecker. Let $r=\left(\displaystyle\sum_{i}\tau x_{i}^{2}+ \alpha_{i}x_{i}+\beta_{i}\right)$, 	where $\tau, \, \alpha_{i}, \, \beta_{i} \in \mathbb{R}$. Consider the metric tensor $\bar{g}=\frac{g}{\varphi^{2}(r)}$ and $f(r)$ satisfying the equation
	\begin{eqnarray}\label{spfe}
	f\mathring{R}ic_{\bar{g}}=\mathring{\nabla}^2_{\bar{g}}f,
	\end{eqnarray}
	where $\mathring{R}ic_{\bar{g}}$ and $\mathring{\nabla}^2_{\bar{g}}$ are the Ricci and the Hessian traceless tensors for the metric $\bar{g}$. Therefore, we are considering spherically symmetric static perfect fluid space-times.

First, we will prove the reduced ODE for a spherically symmetric and locally conformally flat perfect fluid space-time satisfying \eqref{spfe} and 
\begin{eqnarray}\label{densitypressure}
	8\pi\mu=\frac{R_{g}}{2}-\Lambda\quad\mbox{and}\quad 8\pi\rho=\left(\frac{n-1}{n}\right)\left[\frac{\Delta_{g}f}{f}-\frac{(n-2)}{2(n-1)}R_{g}\right]+\Lambda.
	\end{eqnarray}
In what follows, we will consider a non-null cosmological constant $\Lambda.$ From now on, we will indicate the metric we are working on to avoid confusion during the upcoming changes.

 We explore the conformal structure of the spatial factor. To start with, we present the Ricci formula (cf. \cite{Besse})
	\begin{eqnarray}\label{conformal_ricci_tensor}
	Ric_{\bar{g}}=\frac{1}{\varphi^2}\left\{(n-2)\varphi\nabla^{2}_{g}(\varphi)+[\varphi\Delta_{g}\varphi-(n-1)|\nabla_{g}\varphi|^2]g\right\},
	\end{eqnarray}
	where $\nabla_{g}\varphi$, $\nabla^{2}_{g}\varphi$ and $\Delta_{g}\varphi$ stand for the gradient, Hessian and Laplacian of $\varphi$ with respect to $g$, respectively. From the above formula, we get that the scalar curvature
	\begin{eqnarray}\label{conformal_scalar_curvature}
	R_{\bar{g}}=(n-1)\left[2\varphi\Delta_{g}\varphi-n|\nabla_{g}\varphi|^2\right].
	\end{eqnarray}
	
	In what follows, we denote by $\varphi_{,{i}}$, $f_{,{i}}$, $\varphi_{,{i}{j}}$ and $f_{,{i}{j}}$ the first and second order derivatives of $\varphi$ and $f$, with respect to $x_{i}$ and $x_{j}$.
	
	Remember that
	$${\nabla}^2_{\bar{g}}(f)_{ij}=f_{,ij}-\displaystyle\sum_{k}\bar{\Gamma}^{k}_{ij}f_{,k},$$
	where $\bar{\Gamma}^{k}_{ij}$ are the Christoffel symbols of the metric $\bar{g}$. For $i,\,j,\,k$ distinct, we have
	$$\bar{\Gamma}^{k}_{ij}=0,\quad\bar{\Gamma}^{i}_{ij}=-\frac{\varphi_{,j}}{\varphi},
	\quad\bar{\Gamma}^{k}_{ii}=\frac{\varphi_{,k}}{\varphi},\quad\bar{\Gamma}^{i}_{ii}=-\frac{\varphi_{,i}}{\varphi}.$$
	Hence,
	\begin{eqnarray}\label{hess i dif j}
	{\nabla}^2_{\bar{g}}(f)_{ij} = f_{,ij} + \frac{\varphi_{,i}f_{,j}}{\varphi} + \frac{\varphi_{,j}f_{,i}}{\varphi};\quad i\neq j.
	\end{eqnarray}
	Similarly, by considering $i=j$, we have that
	$${\nabla}^2_{\bar{g}}(f)_{ii} = f_{,ii} + 2\frac{\varphi_{,i}f_{,i}}{\varphi} - \displaystyle\sum_{k}\frac{\varphi_{,k}}{\varphi}f_{,k}.$$
	
	Now, considering that $\bar{g}=\frac{1}{\varphi^{2}}g$ is a conformal metric where $g=\delta_{ij}$. We are assuming that $\varphi(r)$ and $f(r)$ are functions of $r$. Hence, we have
	\begin{eqnarray}\label{derivada1}
	\varphi_{,{i}}=\varphi'r_{,i};\quad \varphi_{,{i}{j}}=\varphi''r_{,i}r_{,j}+\varphi'r_{,ij};\quad\varphi_{,{i}{i}}
	=\varphi''r_{,i}^{2}+\varphi'r_{,ii}
	\end{eqnarray}
	and
	\begin{eqnarray}\label{derivada2}
	f_{,{i}}=f'r_{,i};\quad f_{,{i}{j}}=f''r_{,i}r_{,j}+f'r_{,ij};\quad f_{,{i}{i}}
	=f''r_{,i}^{2}+f'r_{,ii},
	\end{eqnarray}
	where $$\varphi'=\frac{d\varphi}{dr}\quad\mbox{and}\quad f'=\frac{df}{dr}.$$

	The following result provides the PDE for a static perfect fluid that is locally conformal to an Euclidean space.

	\begin{lemma}\label{lemma1}
		Let $(\mathbb{R}^{n}, g)$ be an Euclidean space, $n\geq3$, with Cartesian coordinates $(x_{1},\ldots , x_{n})$ and metric components $g_{ij} = \delta_{ij}$. Let $O\subset\mathbb{R}^{n}$ be open and $f :O\rightarrow (0,+\infty)$ be smooth. Then there exists a metric $\bar{g} = \frac{g}{\varphi^{2}}$ on $O$ for which $(O, \bar{g}, f)$ satisfies the static perfect fluid equation (\ref{spfe}) if, and only if, the functions $\varphi$ and $f$ satisfy 
		\begin{equation}\label{41}
		(n-2)f\varphi_{,{i}{j}}-\varphi{f}_{,{i}{j}} - \varphi_{,{i}}f_{,{j}}- \varphi_{,{j}}f_{,{i}}=0,\quad i\neq j
		\end{equation}
		and for each $i=j$
		\begin{eqnarray}\label{42}
		n[(n-2)f\varphi_{,ii}-\varphi{f}_{,ii}-2\varphi_{,i}f_{,i}]=\displaystyle\sum_{k}\left[(n-2)f\varphi_{,kk}-\varphi{f}_{,kk}-2\varphi_{,k}f_{,k}\right].
		\end{eqnarray}
	\end{lemma}

	\begin{proof}
	Consider $i\neq j$. Hence, from \eqref{spfe} we have
	\begin{eqnarray*}
		f\left(Ric_{\bar{g}}\right)_
		{ij}=\nabla^{2}_{\bar{g}}(f)_{ij}.
	\end{eqnarray*}
	Then, from \eqref{conformal_ricci_tensor} we obtain that
	\begin{eqnarray*}
		\left(Ric_{\bar{g}}\right)_{ij}=(n-2)\frac{\varphi_{,ij}}{\varphi}.
	\end{eqnarray*}
	Thus, from the above two equations and from the Hessian formula for a conformal metric given by \eqref{hess i dif j}, we obtain \eqref{41}.

	Now, from \eqref{conformal_ricci_tensor} and \eqref{conformal_scalar_curvature} we have
	\begin{eqnarray}\label{conformal_trace_free_ricci_tensor}
	\left(Ric_{\bar{g}}-\frac{R_{\bar{g}}}{n}\bar{g}\right)_{ii}&=&\frac{(n-2)}{\varphi}\left(\nabla^{2}_{g}\varphi-\frac{\Delta_{g}\varphi}{n}g\right)_{ii}.
	\end{eqnarray}
	In a local coordinate system, we have that \eqref{conformal_trace_free_ricci_tensor} became 
	\begin{eqnarray}\label{conformal_trace_free_ricci_tensor_coordinates}
	\left(Ric_{\bar{g}}-\frac{R_{\bar{g}}}{n}\bar{g}\right)_{ii}=\frac{(n-2)}{\varphi}\left(\varphi_{,ii}-\frac{1}{n}\displaystyle\sum_{k}\varphi_{,kk}\right).
	\end{eqnarray}
	On the other hand, for $i=j$, 
	\begin{eqnarray}\label{Hessian_traceless_tensor}
	\nabla^{2}_{\bar{g}}(f)_{ii}-\frac{\Delta_{\bar{g}}f}{n}\bar{g}_{ii}=f_{,ii}+2\frac{\varphi_{,i}f_{,i}}{\varphi}-\displaystyle\sum_{k}\left[\frac{f_{,kk}}{n}-\frac{(n-2)}{n}\left(\frac{\varphi_{,k}f_{,k}}{\varphi}\right)+\frac{\varphi_{,k}f_{,k}}{\varphi}\right].
	\end{eqnarray}
	Then, \eqref{spfe}, \eqref{conformal_trace_free_ricci_tensor_coordinates} and \eqref{Hessian_traceless_tensor} give us \eqref{42}.
	\end{proof}

The theorem below will prove that even in locally conformally flat static perfect fluid spaces, the chance of getting examples is too high. Therefore, obtaining classification results in the conformally flat case is still relevant.

 \begin{theorem}\label{theorem5}
		Let $(\mathbb{R}^{n}, g)$ be a Euclidean space, $n\geq3$, with Cartesian coordinates $(x_{1},\ldots , x_{n})$ and metric components $g_{ij} = \delta_{ij}$. Let $O\subset\mathbb{R}^{n}$ be open and $f :O\rightarrow (0,+\infty)$ be smooth function. Then, there exists a metric $\bar{g} = \frac{g}{\varphi^{2}}$ on $O$ for which $(O, \bar{g}, f)$ satisfies the static perfect fluid equation (\ref{spfe}) such that $f(r)$ and $\varphi(r)$ are invariant under an $(n-1)$-dimensional orthogonal group whose basic invariant is $r$ if, and only if,
		\begin{eqnarray}\label{M1}
		(n-2)f\varphi''-f''\varphi-2\varphi'f'=0.
		\end{eqnarray}
		Moreover, 
		\begin{eqnarray}\label{den}
		8\pi\mu&=&\frac{(n-1)}{2}\left\{4n\tau\varphi\varphi'+[2\varphi\varphi''-n(\varphi')^{2}](4\tau r+ C)\right\}-\Lambda
		\end{eqnarray}
		and
		\begin{eqnarray}\label{press}
		8\pi\rho&=&\left(\frac{n-1}{n}\right)\Bigg\{\left[(n-2)\varphi\varphi''-\frac{n}{2}(n-2)(\varphi')^{2}-n\varphi\varphi'\frac{f'}{f}\right](4\tau r+ C)\nonumber\\
		&+&2n\tau\frac{\varphi}{f}\left[f'\varphi -(n-2)f\varphi'\right]\Bigg\}+\Lambda
		\end{eqnarray}
		stands for the energy density and pressure, respectively. Here,  $C=\displaystyle\sum_{i} (\alpha_{i}^2 -4\tau \beta_{i})$.
	\end{theorem}

 \begin{remark}
In other words, the above theorem shows that you will get an ODE if you pick either $f$ or $\varphi$ and put it in the equation \eqref{M1}. By integration (if possible), we get a static stellar model.
 \end{remark}

 	\noindent {\bf Proof of Theorem \ref{theorem5}:}
	Since the basic invariant is of the form $r=\sum_{k} U_{k}(x_{k})$, where $U_{k}(x_{k})=\tau x_k^2+\alpha_k x_k + \beta_k$, from \eqref{derivada1} and \eqref{derivada2} and \eqref{41} we have 
	\begin{eqnarray}\label{eq111}
	\left[(n-2)f\varphi''-f''\varphi-2\varphi'f'\right]r_{,i}r_{,j}+[(n-2)f\varphi'-\varphi f']r_{,ij}=0;\quad\forall\, i\neq j.
	\end{eqnarray}
 Considering
	\begin{eqnarray}\label{eq-for-F}
	F(r)=\frac{r_{,ij}}{r_{,i}r_{,j}}
	\end{eqnarray}
	  from \eqref{eq111} and \eqref{eq-for-F} we have 
	\begin{eqnarray}\label{bita}
	\left[(n-2)f\varphi''-f''\varphi-2\varphi'f'\right]+F[(n-2)f\varphi'-\varphi f']=0
	\end{eqnarray}
 which provides \eqref{M1}. Moreover,  
	\begin{eqnarray}\label{derivatives}
	\displaystyle\sum_{k}r_{,k}^{2}&=&\displaystyle\sum_{k}(U^{\prime}_{k})^{2}= 	\displaystyle\sum_{k}\left(2\tau x_{x}+\alpha_{k}^{2}\right)= 	\displaystyle\sum_{k}\left(4\tau^{2}x_{k}^{2}+4\tau\alpha_{k}x_{k}+\alpha_{k}^{2}\right) \nonumber \\
	&=&4\tau\left(	\displaystyle\sum_{k}\tau x_{k}^{2}+\alpha_{k}x_{k}+\beta_{k}\right) - \displaystyle\sum_{k}4\tau\beta_{k}+\displaystyle\sum_{k}\alpha_{k}^{2} \nonumber \\
	&=& 4\tau r+C
	\end{eqnarray}
 and
	\begin{eqnarray}
\displaystyle\sum_{k}r_{,kk}=\displaystyle\sum_{k}U_{k}^{\prime\prime} =	\displaystyle\sum_{k}\left(2\tau\right)=2n\tau,
	\end{eqnarray}
	where $C=\displaystyle\sum_{k} (\alpha_{k}^{2} -4\tau \beta_{k})$.
	
	Let us prove the formulae for energy density and pressure of a static perfect fluid space. From \eqref{conformal_scalar_curvature} and \eqref{derivatives} we have
	\begin{eqnarray}\label{scalar_curvature_pseudo-orthogonal}
	R_{\bar{g}}&=&(n-1)\left(2\varphi\displaystyle\sum_{k}\varphi_{,kk}-n\displaystyle\sum_{k}\varphi_{,k}^{2}\right)\nonumber\\
	&=&(n-1)\left\{2\varphi\displaystyle\sum_{k}(\varphi''r_{,k}^{2}+\varphi'r_{,kk})-n(\varphi')^{2}\displaystyle\sum_{k}r_{,k}^{2}\right\}\nonumber\\
	&=&(n-1)\left\{2\varphi\varphi'\displaystyle\sum_{k}r_{,kk}+[2\varphi\varphi''-n(\varphi')^{2}]\displaystyle\sum_{k}r_{,k}^{2}\right\}\nonumber\\
	&=&(n-1)\left\{4n\tau\varphi\varphi'+[2\varphi\varphi''-n(\varphi')^{2}](4\tau r+ C)\right\}.
	\end{eqnarray}
	Therefore, it is easy to see from \eqref{densitypressure} that the energy density is given by \eqref{den}, i.e.,
 from (\ref{scalar_curvature_pseudo-orthogonal}) we have
	\begin{eqnarray*}    	8\pi\mu&=&\frac{\left(n-1\right)}{2}\left\{4n\tau\varphi\varphi' +\left(4\tau r+\Lambda\right)\left(2\varphi\varphi''-n\left(\varphi'\right)^{2}\right) \right\}-\Lambda.                                        \end{eqnarray*}
	
	Moreover, it is well-known that 
	\begin{eqnarray*}
		\nabla^{2}_{\bar{g}}(f)_{ii}=f_{,ii}+2\frac{\varphi_{,i}f_{,i}}{\varphi}-\displaystyle\sum_{k}\frac{\varphi_{,k}f_{,k}}{\varphi}.
	\end{eqnarray*}
	Therefore, by definition we have that $\Delta_{\bar{g}}f=\displaystyle\sum_{k}\varphi^{2}\nabla^{2}_{\bar{g}}(f)_{kk}$ which implies that
	\begin{eqnarray}\label{laplacian_pseudo-orthogonal}
	\Delta_{\bar{g}}f&=&\displaystyle\sum_{k}\varphi^{2}\left[f_{,kk}-(n-2)\frac{\varphi_{,k}f_{,k}}{\varphi}\right]\nonumber\\
	&=&\displaystyle\sum_{k}\varphi^{2}\left[f''r_{,k}^{2}+f'r_{,kk}-(n-2)\frac{\varphi'f'}{\varphi}r_{,k}^{2}\right]\nonumber\\
	&=&[f''\varphi^{2}-(n-2)\varphi\varphi'f']\displaystyle\sum_{k}r_{,k}^{2}+\varphi^{2}f'\displaystyle\sum_{k}r_{,kk}\nonumber\\
	&=&[f''\varphi^{2}-(n-2)\varphi\varphi'f'](4\tau r+ C)+2n\tau\varphi^{2}f'.
	\end{eqnarray}
	
	Now, from \eqref{densitypressure}, \eqref{scalar_curvature_pseudo-orthogonal} and \eqref{laplacian_pseudo-orthogonal} we obtain
	\begin{eqnarray*}
		8\pi\rho&=&\left(\frac{n-1}{n}\right)\Big\{[\frac{f''}{f}\varphi^{2}-(n-2)\varphi\varphi'\frac{f'}{f}](4\tau r+ C)+2n\tau\varphi^{2}\frac{f'}{f}\nonumber\\
		&-&\frac{(n-2)}{2}\left\{4n\tau\varphi\varphi'+[2\varphi\varphi''-n(\varphi')^{2}](4\tau r+ C)\right\}\Big\}+\Lambda\nonumber\\
		&=&\left(\frac{n-1}{n}\right)\Big\{[\frac{f''}{f}\varphi^{2}-(n-2)\varphi\varphi'\frac{f'}{f}](4\tau r+ C)+2n\tau\varphi^{2}\frac{f'}{f}\nonumber\\
		&-&2(n-2)n\tau\varphi\varphi'+[(n-2)\varphi\varphi''-\frac{n}{2}(n-2)(\varphi')^{2}](4\tau r+ C)\Big\}+\Lambda\nonumber\\
		&=&\left(\frac{n-1}{n}\right)\Bigg\{\Big[(n-2)\varphi\varphi''-\frac{n}{2}(n-2)(\varphi')^{2}+\frac{f''}{f}\varphi^{2}-(n-2)\varphi\varphi'\frac{f'}{f}\Big](4\tau r+ C)\nonumber\\
		&+&2n\tau\varphi\Big[\varphi\frac{f'}{f}-(n-2)\varphi'\Big]\Bigg\}+\Lambda\nonumber\\
		&=&\left(\frac{n-1}{n}\right)\Bigg\{\left[(n-2)\varphi\varphi''-\frac{n}{2}(n-2)(\varphi')^{2}-n\varphi\varphi'\frac{f'}{f}\right](4\tau r+ C)\nonumber\\
		&+&2n\tau\frac{\varphi}{f}\left[f'\varphi -(n-2)f\varphi'\right]\Bigg\}+\Lambda.\nonumber\\
	\end{eqnarray*}

	\hfill $\Box$

	\begin{theorem}\label{coro11111}
		Let $M=I\times\mathbb{S}^{n-1}$, where $I\subset\mathbb{R},$ be a Riemannian manifold with metric tensor $g=dr^2+\tanh(r)^2\,g_{\mathbb{S}^{n-1}}$, where $g_{\mathbb{S}^{n-1}}$ stands by the canonical metric of $\mathbb{S}^{n-1}$ and $r\in I$. Then, $(M^n,\,g,\,f)$ is a static perfect fluid space with lapse function $$f(r)=
  A\sin\left(\sqrt{n-2}\log\cosh\,r\right)+B\cos\left(\sqrt{n-2}\log\cosh\,r\right);\quad\mbox{where}\quad A,\,B\in\mathbb{R}.$$  
  Moreover, the static perfect fluid space $(M^n,\,g,\,f)$ has positive sectional curvature, and the metric $g$ extends to a smooth metric through the origin.
	\end{theorem}

	\begin{proof}[{\bf Proof of Theorem \ref{coro11111}.}]
 Let $\varphi : \mathbb{R}^n \rightarrow \mathbb{R}$ and $f : \mathbb{R}^n \rightarrow \mathbb{R}$, $n\geq3$, be functions given by
		\begin{eqnarray} \label{conformal-particular1}
		\varphi(x_{1},\ldots,x_{n})=\left[\displaystyle\sum_{\ell=1}^{n}x_{\ell}^{2}+1\right]^{1/2}
		\end{eqnarray}
  and
		\begin{eqnarray} \label{lapse-particular1}
		f(x_1,\,x_2,\,\ldots,\,x_n)&=&A\sin\left(\frac{\sqrt{n-2}}{2}\log(\displaystyle\sum_{\ell=1}^{n}x_{\ell}^{2}+1)\right)+B\cos\left(\frac{\sqrt{n-2}}{2}\log(\displaystyle\sum_{\ell=1}^{n}x_{\ell}^{2}+1)\right).
		\end{eqnarray}
		Then $\left(\mathbb{R}^n,\,\dfrac{\delta_{ij}}{\varphi^2},\,f \right)$ is a static perfect fluid space with positive sectional curvature. Here, \,$A,\,B\in\mathbb{R}$ and $\delta_{ij}$ is the delta Kronecker.
  
	In order, we are free to choose any conformal function we want in Theorem \ref{theorem5}. Moreover, we can choose $\tau=1$ and $\alpha_{i}=\beta_{i}=0$ for all $i$. So, we consider $\varphi$ given by \eqref{conformal-particular1}, and then we apply this function into \eqref{M1} to obtain \eqref{lapse-particular1}.
	
    In fact, choosing the function, $\varphi(r) = \sqrt{1+r^{2}}$, replacing into (\ref{M1}) we have
    \begin{equation}\label{fsolution}
     4(r^{2} +1)^{2}f''-4(r^{2}+1)f'+ (n-2)f=0.
	\end{equation}
	This is an Euler equation. With an appropriate change of variables, we can transform (\ref{fsolution}) into an ODE with real coefficients whose characteristic equation has complex roots with solution given by
	\begin{equation}\label{freal}
	   f(r) = A\sin\left(\frac{\sqrt{n-2}}{2}\log(1+r^{2})\right)
		+B\cos\left(\frac{\sqrt{n-2}}{2}\log(1+r^{2})\right);\quad\mbox{where}\quad A,\,B\in\mathbb{R}.
	\end{equation}
Making $r^{2} = \displaystyle\sum_{k}x_{k}^{2}$ we obtain (\ref{conformal-particular1}) and (\ref{lapse-particular1}).
	
	The expression for the sectional curvatures follows from a standard and straightforward computation regarding conformally flat metrics. In fact, the sectional curvature for a metric $\bar{g}=\frac{\delta_{ij}}{\varphi^{2}}$ is given by
	\begin{eqnarray}\label{curvaturesectionalconformalmetric}
	\mathcal{K}_{ij}&=&\varphi^{2}\Bigg\{(\log\varphi)_{,ii}+(\log\varphi)_{,jj}-\displaystyle\sum_{\ell \neq i,j}[(\log\varphi)_{,\ell}]^{2}\Bigg\}.
	\end{eqnarray}
	
	Since $\log(\varphi) = \dfrac{1}{2}\log \left( \displaystyle \sum_{\ell} x_{\ell}^2 + 1 \right)$, we have
	\begin{equation}
	\begin{array}{rcl}
	(\log\varphi)_{,i} &=& \dfrac{x_i}{\left(\displaystyle \sum_{\ell} x_{\ell}^2 + 1\right)}, \\
	(\log\varphi)_{,ii} &=& \dfrac{\left(\displaystyle \sum_{\ell} x_{\ell}^2 + 1-2x_i^2\right)}{\left(\displaystyle \sum_{\ell} x_{\ell}^2 + 1\right)^2}.
	\end{array} \label{ln-varphi-derivatives}
	\end{equation}
	Therefore, from (\ref{curvaturesectionalconformalmetric}) and \eqref{ln-varphi-derivatives} we obtain
	
	\begin{eqnarray*}
		0<\mathcal{K}_{ij}&=&\frac{1}{\left(\displaystyle \sum_{\ell} x_{\ell}^2 + 1\right)}\Bigg\{\displaystyle \sum_{\ell\neq i,j} x_{\ell}^2+2\Bigg\}.
	\end{eqnarray*}

	Through the introduction of polar coordinates on $(\mathbb{R}^n \setminus \left\{0 \right\}, \bar{g})$, with $\bar{g} = \dfrac{\delta_{ij}}{\varphi^2}$, we have
	\begin{eqnarray*}
		\bar{g}=(s^{2}+1)^{-1}[ds^{2}+s^{2}g_{\mathbb{S}^{n-1}}], \,\, s > 0,
	\end{eqnarray*}
	where $g_{\mathbb{S}^{n-1}}$ stands for the canonical metric on the sphere of radius $1$.
	Considering the change $s=\sinh(r)$ we get that
	\begin{eqnarray*}
		\bar{g}&=&[\cosh^{2}(r)]^{-1}[\cosh^{2}(r)dr^{2}+\sinh^{2}(r)g_{\mathbb{S}^{n-1}}]\nonumber\\
		&=&dr^{2}+\tanh^{2}(r)g_{\mathbb{S}^{n-1}}.
	\end{eqnarray*}

	Moreover, $\bar{g}=dr^{2}+\tanh^{2}(r)g_{\mathbb{S}^{n-1}}$ extends to a smooth metric through the origin (cf. Lemma A.2 in \cite{chow1}). 
\end{proof}

	\begin{example}\label{extop}
		For the space-time provide in Theorem \ref{coro1}, consider $n=3$, $A=1$ and $B=0$. Then, from the introduction of spherical coordinates, we have
		\begin{eqnarray*}
			\hat{g}=-f^{2}(r)dt^{2}
			+dr^{2}+\phi^{2}(r)g_{\mathbb{S}^{2}},
		\end{eqnarray*}
		in which
		$$f(r)=\sin\left(\log(\cosh r)\right)\quad\mbox{and}\quad\phi(r)=\tanh(r).$$ 
		Here $g_{\mathbb{S}^{2}}$ stands for the canonical metric on the sphere of radius $1$.

		Using polar coordinates and making a change of variables to $\sinh$, from \eqref{conformal-particular1} and \eqref{lapse-particular1} we have
		\begin{eqnarray*}
			\varphi(r)=\cosh(r),
		\end{eqnarray*}
		and
		\begin{eqnarray*}
			f(r)=\sin\left(\log(\cosh(r))\right).
		\end{eqnarray*}

		Considering \eqref{den}, \eqref{press} and the above equations we get
		\begin{eqnarray*}
			8\pi\mu=\frac{\cosh^{2}(r)+5}{\cosh^{2}(r)}-\Lambda
		\end{eqnarray*}
		and
		\begin{eqnarray*}
			8\pi\rho=-\left[\frac{11}{3}\tanh^{2}(r)+\frac{2}{\cosh^{2}(r)}-\frac{2}{\cosh^{2}(r)(\tan\circ\log\circ\cosh)(r)}\right]+\Lambda.
		\end{eqnarray*}

		Making the change of variable $\tilde{r}=\mathfrak{M}\cosh^{2}(r)$ (see \cite{chow2}), in which $\mathfrak{M}$ represent the mass, we have 
		\begin{eqnarray*}\label{pfbh}
		\hat{g}=-\left[(\sin\circ\log)\left(\sqrt{\frac{\tilde{r}}{\mathfrak{M}}}\right)\right]^{2}dt^{2}
		+\frac{1}{4\tilde{r}^{2}}\left(1-\frac{\mathfrak{M}}{\tilde{r}}\right)^{-1}d\tilde{r}^{2}+\left(1-\frac{\mathfrak{M}}{\tilde{r}}\right)g_{\mathbb{S}^{2}},
		\end{eqnarray*}
		where the energy density $\mu$ and the pressure $\rho$ of such space-time are, respectively, given by
		\begin{eqnarray*}\label{pqr}
		8\pi\mu=1+\frac{5\mathfrak{M}}{\tilde{r}}-\Lambda
		\end{eqnarray*}
		and
		\begin{eqnarray*}\label{pqs}
		8\pi\rho=\frac{2\mathfrak{M}}{\tilde{r}(\tan\circ\log)(\sqrt{\frac{\tilde{r}}{\mathfrak{M}}})}+\frac{5\mathfrak{M}}{3\tilde{r}}-\frac{11}{3}+\Lambda.
		\end{eqnarray*}
	\end{example}

%%%%%%%%%%%%%%%%%%%%%%%%%%%%%%%%%%%%%%%%%%%%%%%%%%%%%%%%%%%%%%%%%%%%%%%%%%%%%%%%%%%%%%%%%%%%%%%%%%%%%%%%%%%%%%%%%%%%%%%%%%%%%%%%%%%%%%5

\noindent{\bf Acknowledgement:}
The authors sincerely thank Professor João Paulo dos Santos for discussing some properties of our solution in Theorem \ref{coro11111}. The authors are grateful to Professor Ernani Ribeiro Jr. for fruitful discussions about some results present in this work.

 \

\noindent{\bf Conflict of interest:} The authors declare no conflict of interest.

\

\noindent{\bf Data Availability:} Not applicable.

\
%%%%%%%%%%%%%%%%%%%%%%%%%%%%%%%%%%%%%%%%%%%%%%%%%%%%%%%%%%%%%%%%%%%%%%%%%%%%%%%%%%%%%%%%%%%%%%%%%%%%%%%%%%%%%%%%%%%%%%%%%%%%%%%%%%%%%%%%


\begin{thebibliography}{99}

%------------------
% A
%-----------------

%\bibitem{ambrozio2017} {Ambrozio, Lucas.} {\em On static three-manifolds with positive scalar curvature.} J. Differential Geom. 107 (2017), no. 1, 1–45. MR3698233

%\bibitem{anderson1999}{Anderson, Michael T.} {\em Scalar curvature, metric degenerations and the static vacuum Einstein equations on 3-manifolds. I.} Geom. Funct. Anal. 9 (1999), no. 5, 855–967. MR1726233

%\bibitem{anderson1989}{Anderson, Michael T.; Rodríguez, Lucio.} {\em Minimal surfaces and 3-manifolds of nonnegative Ricci curvature}. Math. Ann. 284 (1989), no. 3, 461–475. MR 1001714

%------------------
% B
%------------------



\bibitem{ambrozio2015} {Ambrozio, L.} - {\em Rigidity of area-minimizing free boundary surfaces in mean convex three-manifolds.} J. Geom. Anal. 25 (2015), no. 2, 1001–1017. MR3319958

\bibitem{ambrozio2017} {Ambrozio, L.} - {\em On static three-manifolds with positive scalar curvature.} J. Diff. Geom. 107.1 (2017): 1-45. MR3698233

\bibitem{barbosa2021} {Barbosa, E.; Espinar, J. M.} - {\em On free boundary minimal hypersurfaces in the Riemannian Schwarzschild space.} J. Geom. Anal. 31 (2021), no. 12, 12548-12567. MR4322577

\bibitem{barboza2018} {Barboza, M.; Leandro, B.; Pina, R.} - {\em Invariant solutions for the Einstein field equation.} J. Math. Phys. 59 (2018), no. 6, 062501, 9 pp. MR3813916

\bibitem{Besse} {Besse, A. L.} - {\em Einstein Manifolds.} Reprint of the 1987 edition. Classics in Mathematics. Springer-Verlag, Berlin, 2008. xii+516 pp. ISBN: 978-3-540-74120-6. MR2371700

%\bibitem{besseieres} {Bessi\`eres, L; Lafontaine, J.; Rozoy, L.} {\em Courbure scalaire et trous noirs.} S\'emin. Th\'eor. Spectr. G\'eom., 18, Univ. Grenoble I, Saint-Martin-d'H\`eres, 2000.



\bibitem{bethuel2006}{Bethuel, F., et al.}  {\em Calculus of Variations and Geometric Evolution Problems.} Lectures given at the 2nd Session of the Centro Internazionale Matematico Estivo (CIME) held in Cetaro, Italy, June 15-22, 1996. Springer, (2006). MR1730218

%\bibitem{beigsimon1992}{Beig, R.; Simon, W.}  {On the uniqueness of static perfect-fluid solutions in general relativity.} Commun. Math. Phys. 144 (1992), no. 2, 373-390. MR1152378

\bibitem{Brendle2013}{Brendle, S.} {\em Constant mean curvature surfaces in warped product  manifolds.} Publ. Math. Inst. Hautes \'Etudes Sci. 117 (2013),
247-269. MR3090261  


\bibitem{vasquez}{Brozos-Vázquez, M.; Garc\'ia-R\'io, E.; V\'azquez-Lorenzo, R.} {\em Some remarks on locally conformally flat static space-times.} J. Math. Phys. 46, 022501 (2005); doi: 10.1063/1.1832755

%\bibitem{Leandro} {Leandro, B.} {\em A note on critical point metrics of the total scalar curvature functional.} J.  Math. Anal. Appl. 424 (2015), 1544-1548.


%------------------
% C
%------------------

%\bibitem{cao}{Cao, H-D., Sun, X. \& Zhang, Y.} {\em On the structure of gradient Yamabe soltions.} Math. Res. Lett. 19 (2012), no. 04, 767–774. MR3008413


%\bibitem{Cao} Cao, H.-D. and  Chen, Q.: On Bach-flat gradient shrinking Ricci solitons. {\it Duke Math. J.} 162 (2013) 1149-1169.

%\bibitem{carmo2015}{Carmo, Manfredo P.}  {\em Geometria Riemanniana.} Coleção Projeto Euclides. IMPA, 5ª edição, Rio de Janeiro, 2015.  

%\bibitem{caminha2014}{Caminha, Antonio; Neto, Muniz.}  {\em Tópicos de Geometria Diferencial.} Coleção Fronteiras da Matemática. SBM, 1ª edição, Rio de Janeiro, 2014.

%\bibitem{carlotto2016} {Carlotto, A.}  {\it Rigidity of stable minimal hypersurfaces in asymptotically flat spaces.} Calc. Var. Partial Differential Equations 55 (2016), no. 3, Art. 54, 20 pp. MR3500292

\bibitem{cabrerap2017} {Cabrera Pacheco, A. J.; Cederbaum, C.; McCormick, S.; Miao, P.} {\em Asymptotically flat extensions of CMC Bartnik data.} 2017 Class. Quantum Grav. 34 105001.

\bibitem{carrol}{Carroll, S.} {\em Spacetime and Geometry: An Introduction to General Relativity}. Addison Wesley, San Francisco, 2004. ISBN: 0-8053-8732-3 83-01. MR2329798


\bibitem{cederbaum2021} {Cederbaum, C.; Galloway, G.}  {\it Photon surfaces with equipotential time slices.} J. Math. Phys. 62 (2021), no. 3, Paper No. 032504, 22 pp. MR4236797

\bibitem{coutinho} {Coutinho, F.; Diógenes, R.; Leandro, B.; Ribeiro, E., Jr.}  {\it Static perfect fluid space-time on compact manifolds.} Classical Quantum Gravity 37 (2020), no. 1, 015003, 23 pp. MR4054632 

\bibitem{coutinho2021} {Coutinho, F.; Leandro, B.}  {\em On the fluid ball conjecture.} Annals of Global Analysis and Geometry. 


\bibitem{costa2023}{Costa, J.; Diógenes, R.; Pinheiro, N.; Ribeiro Jr, E.} {\em Geometry of static perfect fluid space-time.} 2023 Class. Quantum Grav. 40 205012.



%\bibitem{Corvino2000}{Corvino, Justin.} {\em Scalar curvature deformation and a gluing construction  for the {E}instein constraint equations.} Comm. Math. Phys. 214  (2000), no. 1, 137-189. MR1794269

%\bibitem{Carlotto-Chodosh-Eichmair2016}{Carlotto, Alessandro ;   Chodosh, Otis; Eichmair, Michael .} {\em Effective   versions of the positive mass theorem}, Invent. Math. 206 (2016), no. 3, 975-1016. MR3573977 

\bibitem{chow2}{Chow, B.; Lu, P.; Ni, L.} {\em Hamilton's Ricci flow.} Grad. Stud. Math., 77. American Mathematical Society, Providence, RIScience Press Beijing, New York, 2006, xxxvi+608 pp. MR2274812

\bibitem{chow1} Chow, B., et al.: {\em The Ricci flow: techniques and applications.} American Mathematical Society, 2007.


\bibitem{tiarlos} {Cruz, T.; Lima, V.; de Sousa, A.} {\em Min-max minimal surfaces, horizons and electrostatic systems}, to appear in Jour. Diff. Geom., (arXiv:1912.08600), 1–47.

%--------------------
% D
%--------------------

%\bibitem{doeleman2019}{Doeleman, Sheperd S.} {\em First M87 Event Horizon Telescope Results: I-VI.} The Astrophysical Journal Letters, 875:L1 (17pp), 2019.



%------------------
% E
%------------------

%\bibitem{ellis} {Ellis, G. F. R.; Maartens, Roy; MacCallum, M.A.H.} {\em Relativistic Cosmology.}  Cambridge University Press, New York, 2012. ISBN 978-0-521-38115-4. MR3186049

%\bibitem{frankel} {Frankel, T.; Galloway, G.} {\em Stable minimal surfaces and spatial topology in general relativity.} Math. Z. 181 (1982), no. 3, 395–406. (Reviewer: F. J. Flaherty) 53A10 (53C80 83C99)

%-----------------------------
%F
%-----------------------------
\bibitem{fang2019}{Fang, Yi; Yuan, Wei} {\em Brown-York mass and positive scalar curvature II: Besse's conjecture and related.} Ann. Global Anal. Geom. 56 (2019), no. 1, 1-15. MR3962023

%------------------
% G
%------------------

\bibitem{galloway1993}{Galloway, G.} {\em On the topology of black holes.} Comm. Math. Phys.  151 (1993), no. 1, 53-66. MR1201655

\bibitem{gorini} {Gorini, V.; Moschella, U.; Kamenshchik, A. Yu.; Pasquier, V.; Starobinsky, A. A.} {\em Tolman-Oppenheimer-Volkoff equations in the presence of the Chaplygin gas: Stars and wormholelike solutions.} Phys. Rev. D 78, 064064 (2008).

%------------------
% H
%------------------


%\bibitem{CP3} {Hwang, S.} {\em Critical points of the total scalar curvature functional on the space of metrics of constant scalar curvature}, Manuscripta Math. 103 (2000), no. 2, 135-142.

\bibitem{huang2018} {Huang, L.-H.; Martin, D.; Miao, P.} {\em Static potentials and area minimizing hypersurfaces}. Proc. Amer. Math. Soc. 146 (2018), no. 6, 2647–2661. MR3778165.


\bibitem{jahns}{Jahns, S.} {\em Photon sphere uniqueness in higher-dimensional electrovacuum spacetimes}. Classical and Quantum Gravity 36 (2019), no. 23, 235019.

%------------------
% K
%------------------


%\bibitem{Obata} {Kobayashi, Osamu; Obata, Morio.}  {\em Conformally-flatness and static space-time. Manifolds and Lie Groups.} (Notre Dame, Ind., 1980), pp. 197–206, Progr. Math., 14, Birkhäuser, Boston, Mass., 1981. MR0642858 

\bibitem{kunzle} {K\"unzle, H. P.} {\em On the spherical symmetry of a static perfect fluid.} Commun. Math. Phys. 20 (1971), 85-100.
MR0275833 

%--------------------
% L
%-------------------

%\bibitem{lafontaine} {Lafontaine, J.} - \textit{Sur la g\'eom\'etrie d'une g\'en\'eralisation de l'\'equation diff\'erentielle d'Obata}, J. Math. Pures Appliqu\'ees, 62 (1983), 63-72.

\bibitem{lin2016} {Lin, C.-Y.; Sormani, C.} {\em Bartnik's mass and Hamilton's modified Ricci flow.} Ann. Henri Poincar\'e 17 (2016), no. 10, 2783–2800.

\bibitem{leandro2019}{Leandro, B.; Sol\'orzano, N.} {\em Static perfect fluid spacetime with half conformally flat spatial factor.} Manuscripta Math. 160 (2019), no. 1-2, 51–63. MR3983386 


\bibitem{leandroernanipina} {Leandro, B.;  Pina, H.; Ribeiro, E. Jr.} {\em Volume growth for geodesic balls of static vacuum space on 3-manifolds}.  Ann. Mat. Pura Appl. (4) 199 (2020), no. 3, 863–873. MR4102794


\bibitem{leandro2021}{Leandro, B.} {\em Vanishing conditions on Weyl tensor for Einstein-type manifolds.} Pacific J. Math. 314 (2021), no. 1, 99–113. MR4329972

%\bibitem{liu2013} {Liu, G.} {\em 3-manifolds with nonnegative Ricci curvature}. Invent. Math. 193(2), 367–375 (2013). MR 3090181.

%------------------
% M
%------------------



%\bibitem{lindblom}{Lindblom, Lee. \& Masood-ul-Alam, A. K. M.}  {\em On the spherical symmetry of static stellar models.} Comm. math. phys. 162 (1994), no. 1, 123-145. MR1272769


   	\bibitem{miao2005} {Pengzi Miao} - \emph{A remark on boundary effects in static vacuum initial data sets.} Classical Quantum Gravity 22 (2005), no. 11, L53-L59. MR2145225.

\bibitem{miao2020}{Miao, P.; Wang, Y.; Xie, N.} {\em On Hawking mass and Bartnik mass of CMC surfaces.} Math. Res. Lett. 27 (2020), no. 3, 855-885. MR4216572

\bibitem{massod1987} {Masood-ul-Alam, A. K. M.} {\it On spherical symmetry of static perfect fluid spacetimes and the positive-mass theorem.} Classical Quantum Gravity 4 (1987), no. 3, 625-633. MR0884598

\bibitem{massod2} {Masood-ul-Alam, A. K. M.}   {\em The topology of asymptotically Euclidean static perfect fluid space-time.} Comm. Math. Phys. 108 (1987), no. 2, 193-211. MR0875298 

\bibitem{montezuma2021}{Montezuma, R.}  {\em On free-boundary minimal surfaces in the Riemannian Schwarzschild manifold.} Bull. Braz. Math. Soc. (N.S.) 52 (2021), no. 4, 1055–1071. MR4325895 


\bibitem{sun2018}{Sun, J.} {\em Rigidity of Hawking mass for surfaces in three manifolds.} Pacific J.Math. 292 (2018), no.2, 479–504.


\bibitem{shi2019}{Shi, Y.-G.; Sun, J.; Tian, G.; Wei, D.} {\em Uniqueness of the mean field equation and rigidity of Hawking mass.} Calc. Var. Partial Diff. Equations 58 (2019), no.2, 41, 16pp.




%\bibitem{massod} {Masood-ul-Alam, A. K. M.}  {\em On the spherical symmetry of static perfect fluid spacetimes and the positive-mass theorem.} Class. Quantum Grav. 4 (1987), No 3: 625-633. MR0884598

%\bibitem{meeks}{Meeks, William, III; Simon, Leon; Yau, Shing Tung} {\em Embedded minimal surfaces, exotic spheres, and manifolds with positive Ricci curvature.} Ann. of Math. (2) 116 (1982), no. 3, 621–659. MR0678484

%\bibitem{obata} {Obata, M.} {\em Certain conditions for a Riemannian manifold to be isometric with a sphere}, J. Math. Soc. Japan Vol. 14, No. 3, 1962.

\bibitem{massod1} {Masood-ul-Alam, A. K. M.}  {\em Proof that static stellar models are spherical.} Gen. Relativity Gravitation 39 (2007), no. 1, 55-85. MR2322510

\bibitem{oppenheimer} {Oppenheimer, J. R.; Volkof, G. M.} {\it On massive neutron cores.} Phys. Rev. 55, 374 (1939).

\bibitem{oneill1983}{O'Neill, B.} {\it Semi-Riemannian Geometry With Applications to Relativity.} Pure and Applied Mathematics, 103. Academic Press, Inc. [Harcourt Brace Jovanovich, Publishers], New York, 1983. xiii+468 pp. ISBN: 0-12-526740-1. MR0719023 


%\bibitem{jhans}{Jahns, S.} - {\em Photon sphere uniqueness in higher-dimensional electrovacuum spacetimes.} Class. Quantum Grav. 36(23), (2019): 235019.

%\bibitem{miao} {Miao, Pengzi.} {\em A remark on boundary effects in static vacuum initial data sets.} Classical Quantum Gravity 22 (2005), no. 11, L53-L59. MR2145225 

% \bibitem{miao_tam} {Miao, Pengzi; Tam, Luen-Fai.}  {\em Static potentials on asymptotically flat manifolds.} Ann. Henry Poincaré 16 (2015), no. 10, 2239-2264. MR3385979 

%------------------
% O
%------------------

%\bibitem{oneill1983}{O'Neill, Barrett.}  {\em Semi-Riemannian Geometry. With Applications to Relativity.} Academic Press, New York, 1983. ISBN: 0-12-526740-1. MR0719023 
%-----------------------
%Q
%---------------------------
%\bibitem{qing1} Qing, J. and Yuan, W.: A note on static spaces and related problems. {\it J. Geom. Phys}. 74 (2013) 18-27.	
%\bibitem{qing2016} Qing, J. and Yuan, W.: On scalar curvature rigidity of vacuum static spaces. {\it Math. Annalen}.  365 (2016) 1257-1277. MR3521090
%------------------
% S
%------------------



%------------------
% X
%------------------



%\bibitem{xin2003}{Xin, Yuanlong.}  {\em Mininimal Submanifolds and Related Topics.} Nankai Tracts in Mathematics (Vol. 8), World Scientific Publishing, Singapore, 2003. ISBN: 978-981-3236-05-9. MR3837570



%------------------
% W
%------------------


\bibitem{tolman} {Tolman, R. C.} {\em Static solutions of Einstein's field equations for spheres of fluid.} Phys. Rev. 55, 364, (1939).
 

\bibitem{hwang2014} {Yun, G.; Chang, J.; Hwang, S.} -  {\em Total scalar curvature and harmonic curvature.} Taiwan. J. Math. 18(5), (2014): 1439-1458. MR3265071

\bibitem{wald}{Wald, R.} {\em General relativity.}  University of Chicago Press, Chicago, IL, 1984. ISBN: 0-226-87032-4; 0-226-87033-2. MR0757180 

	\bibitem{witten} {Witten, E.} {\em String theory and black holes}. Phys. Rev. D (3) 44 (1991), no. 2, 314-324.

\bibitem{wyman}{Wyman, M.} {\em Radially symmetric distributions of matter.}  Physical Review, 75, no 12, (1949), 1930-1936. 


%%%%%%%%%%%%%%%%%%%%%%%%%%%%%%%%%%%%%%%%%%%%%%%%%%%%%%%%%%%%%%%%%%%%%%%%%%%%%%%%%%%%%%%%%%%%%%%%%%%%%%%%%%%%%%%%%%%%%%%%%%%%%%%%%%%%%%%%%%%%%%%%%%%%%%%%%%%%%%%%%%%%%%%%%%%%%%%%%%%%%%%%%%%%%%%%%%%%%%%%%%%%%%%%%%%%%%%%%%%%%%%%%%%%%%%%%%%%%%%%%%%%%%%%%%%%%%%%%%%%%%%%%%


%\bibitem{barros1} {A. Barros, B. Leandro, and E. Ribeiro Jr} - \textit{Critical metrics of the total scalar curvature functional on 4-manifolds.} Math. Nachr. 288 (2015), 1814-821.

%\bibitem{berger} {M. Berger and D. Ebin} - \textit{Some decompositions of the space of symmetric tensors on a Riemannian manifold.} J. Diff. Geom 3.3-4 (1969): 379-392.







%\bibitem{fischer} {A. E. Fischer and J. E. Marsden} - \textit{Deformations of the scalar curvature.} Duke Math. J. 42.3 (1975), 519-547.



%\bibitem{koiso} {N. KOISO} - \textit{A decomposition of the space of Riemannian metrics on a manifold.} Osaka J. Math. 16 (1979), 423-429.




%\bibitem{obata1} {M. Obata} - \textit{The conjectures on conformal transformations of Riemannian manifolds.} J. Diff. Geom. V 6, N 2 (1971) 247-258.

\end{thebibliography}
\end{document}